\documentclass[11pt]{article}
\usepackage{xcolor,amsmath,mathrsfs,amssymb,accents,graphicx,epstopdf,amsthm,mathtools}
\usepackage{soul,hyperref}
\usepackage{inputenc}
\numberwithin{equation}{section}

\newtheorem{theorem}{Theorem}[section]

\newtheorem{lemma}[theorem]{Lemma}
\newtheorem{definition}[theorem]{Definition}
\newtheorem{remark}[theorem]{Remark}
\newtheorem{proposition}[theorem]{Proposition}
\newtheorem{corollary}[theorem]{Corollary}

\newcommand{\de}{\mathrm{d}}
\newcommand{\bff}{\boldsymbol{f}}
\newcommand{\by}{\boldsymbol{y}}
\newcommand{\bY}{\mathbf{Y}}

\newcommand{\on}{{on} } 

\newcommand{\sfd}{\mathsf d}

\newcommand{\Q}{\mathbb{Q}}
\newcommand{\N}{\mathbb{N}}
\newcommand{\R}{\mathbb{R}}

\newcommand{\Probabilities}[1]{\mathscr P(#1)} 
\newcommand{\Probabilitiesone}[1]{\mathscr P_1(#1)} 
\newcommand{\Measures}{\mathscr M} 
\newcommand{\Measureso}[1]{{\mathscr M}_0({#1})}
\newcommand{\PushForward}[2]{#1_{\#}#2} 

\renewcommand\div{\operatorname{div}}

\newcommand{\norm}[1]{\lVert#1\rVert}

\newcommand{\rmC}{\mathrm C}
 
 \newcommand{\bb}{{\mbox{\boldmath$b$}}}
 \newcommand{\cc}{{\mbox{\boldmath$c$}}}

 \newcommand{\xx}{{\mbox{\boldmath$x$}}}

 \newcommand{\tauV}{{\kern-3pt\tau}}

  \newcommand{\YY}{{\mbox{\boldmath$Y$}}}
 
 \newcommand{\oVVVk}{\overline{\mbox{\boldmath$V$}}\kern-3pt}
 \newcommand{\tVVVk}{\tilde{\mbox{\boldmath$V$}}\kern-3pt}

 \newcommand{\ssigma}{{\mbox{\boldmath$\sigma$}}}

 \newcommand{\eeta}{{\mbox{\boldmath$\eta$}}}

\newcommand{\Lip}{\operatorname{Lip}}
\newcommand{\diam}{\operatorname{diam}}

\newcommand{\crd}{}
\newcommand{\cbl}{}

\newcommand{\nc}{\normalcolor}
\newcommand{\Frechet}{{F-}}
\newcommand{\Gateaux}{{G-}}
\begin{document}

\title{
Spatially Inhomogeneous Evolutionary Games
}
\date{}

\author{Luigi Ambrosio\
 \thanks{Scuola Normale Superiore, Pisa. E-mail: \textsf{luigi.ambrosio@sns.it}}
\and 
 Massimo Fornasier\
\thanks{Technische Universit\"at M\"unchen. E-mail: \textsf{massimo.fornasier@ma.tum.de}}
 \and
 Marco Morandotti\
 \thanks{Technische Universit\"at M\"unchen. E-mail: \textsf{marco.morandotti@ma.tum.de}}
 \and
  Giuseppe Savar\'e\
  \thanks{Universit\`a di Pavia. E-mail: \textsf{giuseppe.savare@unipv.it}}
   }

\maketitle

\begin{abstract}
We introduce and study a mean-field model for a system of spatially distributed players interacting through an evolutionary game driven by a replicator dynamics. 
Strategies evolve by a replicator dynamics influenced by the position and the interaction between different players and return a feedback on the velocity field guiding their motion.

One of the main novelties of our approach concerns the description of the whole system, which can be represented by an evolving probability measure $\Sigma$ \on an infinite dimensional state space (pairs $(x,\sigma)$ of position and distribution of strategies). 
We provide a Lagrangian and a Eulerian description of the evolution, and we prove their equivalence, together with existence, uniqueness, and stability of the solution.
As a byproduct of the stability result, we also obtain convergence of the finite agents model to our mean-field formulation, when the number $N$ of the players goes to infinity, and the initial discrete distribution of positions and strategies converge. 
  
To this aim we develop some basic functional analytic tools to deal with interaction dynamics and continuity equations in Banach spaces, that could be of independent interest.
  
\end{abstract}


\smallskip


\tableofcontents

\section{Introduction}
\subsection{Evolutionary games}

Physical systems naturally tend to minimize the potential energy. For this fundamental reason the study of steady states in physical systems is of utmost relevance, given the expected frequency for such states to occur. This is also the rationale according to which game theorists have focused on the characterization of game equilibria. Very celebrated is in fact the work of John F. Nash \cite{Nash1951}, 
where a new notion of non-cooperative equilibrium is introduced. The main result of \cite{Nash1951},
building on John von Neumann's notion of mixed strategy (necessary to ensure existence of saddle points in zero sum games with two players),
 is the existence of mixed strategy equilibria 
for non-cooperative games with any finite number of players. 
However, already at the origin of game theory, Oskar Morgenstern and John von Neumann pointed out in their classical treatise on game theory \cite{NM1947} the desirability of a ``dynamical'' approach to complement their ``static''
game solution concept.
In fact, while in physical systems evolutions towards minima of the potential energy are explained according to Newton's law (for which evolutions are, for conservative forces, the gradient flows of the potential energy), it is not at all clear whether and how in dynamical games equilibria can emerge. 
Certainly John Nash had anticipated this issue, when (in an
unpublished section of his thesis \cite{Nash1951}) he sketched a ``mass action approach'' to his
equilibrium notion which, many years later, was re-discovered as the evolutionary
approach. 
Additionally, while Nash equilibria are natural ``good'' states for non-cooperative games,
often in cooperative games, such as the one of the prisoner's dilemma, Nash equilibria  are not necessarily the most interesting or favorable states. For this reason, it would be very desirable for a  proper concept of dynamical game to be able to select either Nash equilibria or other type of steady states, according to which is more convenient.  Evolutionary games are dynamical processes describing how the distribution of strategies changes in time according to their individual success.

\subsection{Spatially homogeneous replicator dynamics }
One of most advocated mechanisms of dynamical choice of strategies is
based on a selection principle, inspired by Darwinian evolution
concepts. The main idea is to re-interpret the probability of picking
a certain strategy with the distribution of a
population of players adopting those strategies. The emergence of
steady mixed strategies would be the result of an evolutionary
selection: {\crd at discrete times} players meet randomly, interact
according to their strategies, and obtain a payoff. This payoff
determines how the frequencies in the strategies will evolve.
The discrete time stochastic evolution  described above has been formalized in \cite{BorgersSarin97} and yields,  as an appropriate limit is considered, a continuous time dynamics as follows:
In games where players can adopt pure strategies out of a finite set
of $N$ choices, we may describe those as $u_1, \ldots, u_N \in U$,
where $U$ is the set of strategies. We may denote with $\sigma_i$ the
frequency with which players pick the strategy $u_i$. 
The payoff of playing strategy $u_i$ against $u_j$ will be denoted by $J(u_i,u_j)$, 
where $J\colon U\times U \to \mathbb R$. The relative success of the strategy $u_i$ with respect to the strategies played by the population is measured by
\begin{equation}\label{discrep}
\Delta_N(u_i) = \sum_{j=1}^N J(u_i,u_j)\sigma_j -   \sum_{\ell=1}^N\sum_{j=1}^N \sigma_ \ell J(u_\ell,u_j)\sigma_j, \quad i=1,\dots,N.
\end{equation}
The relative rate of change of usage of the strategy $u_i$ is then described by
$$
\frac{\dot{\sigma}_i}{\sigma_i} = \Delta_N(u_i)= \sum_{j=1}^N J(u_i,u_j)\sigma_j -   \sum_{\ell=1}^N\sum_{j=1}^N \sigma_ \ell J(u_\ell,u_j)\sigma_j, \quad i=1,\dots,N,
$$
or
\begin{equation}\label{eq:hom_repl}
\dot{\sigma}_i =  \left ( \sum_{j=1}^N J(u_i,u_j)\sigma_j -   \sum_{\ell=1}^N\sum_{j=1}^N \sigma_ \ell J(u_\ell,u_j)\sigma_j \right)\sigma_i, \quad i=1,\dots,N.
\end{equation}
The system of ordinary differential equations \eqref{eq:hom_repl} is known as {\it replicator dynamics} in the literature of evolutionary games \cite{HSbook}. It is  one of the most popular dynamical game models, because its $\omega$-limit (the set of accumulation points of the dynamics) and steady states are closely related to the Nash equilibria of the game described by the payoff matrix $A=(J(u_i,u_j))_{ij}$ \cite[Thorem 7.2.1]{HSbook} (the  so-called  ``folk theorem  of  evolutionary  game  theory''). Moreover, as discussed in \cite{HS2003}, adopting an equilibrium-based viewpoint is often unable to always account for the long-term behaviour of realistic players, who adjust their behaviour to maximise their payoff.  The replicator dynamics aims at being a more robust model.

\subsection{Mean-field replicator dynamics}
There is by now a large scope of literature addressing the replicator dynamics for infinite or continuous strategies \cite{Bomze1990,Cressman2005,DJMR2008,Galstyan2013,HOR2009,Norman2008,OR2001,RR2015}, which can be viewed as a natural limit for $N\to \infty$ of   system \eqref{eq:hom_repl}. The way of deducing this limit is by defining the probability measure
$$
\sigma^N_t \coloneqq \sum_{j=1}^N \sigma_{j,t} \delta_{u_j} \in \Probabilities{U},
$$
and its evolution according to 
$$
\dot{\sigma}^N_{t} =  \left ( \int_{U}
J(\cdot ,u') \, \de \sigma^N_t(u') - \int_{U \times U} J(w,u') \, \de
\sigma_t^N(w) \de \sigma_t^N(u') \right) \sigma_t^N, $$
or, in weak form,
$$
\frac{\de}{\de t} \int_U \varphi(u)\, \de \sigma^N_t(u) = \int_U \varphi(u) \left ( \int_{U} J(u,u') \, \de \sigma^N_t(u') - \int_{U \times U} J(w,u') \, \de \sigma^N_t(w) \de \sigma^N_t(u') \right)\, \de \sigma_t^N(u),
$$
for any $\varphi \in \mathrm C(U)$. For any $\sigma \in \Probabilities{U}$, we may denote 
$$
\Delta_\sigma(u) \coloneqq \left ( \int_{U} J(u,u') \, \de \sigma(u') - \int_{U \times U} J(w,u') \, \de \sigma(w) \de \sigma(u') \right),
$$
so that $\Delta_{\sigma}(u_i) = \Delta_N(u_i)$ as in \eqref{discrep} for $\sigma=\sum_{j=1}^N \sigma_j\delta_{u_j}$.
By assuming that the initial conditions $\sigma_0^N \rightharpoonup \bar \sigma$ for a given $\bar{\sigma} \in \Probabilities{U}$, one can show that $\sigma^N_t \rightharpoonup \sigma_t$ for $N \to \infty$ for any $t$, where $\sigma$ is the solution to
$$
\frac{\de}{\de t} \int_U \varphi(u)\, \de \sigma_t(u) = \int_U \varphi(u) \Delta_\sigma(u)\, \de \sigma_t(u), \qquad \sigma(0) = \bar{\sigma}.
$$
A result of well-posedness of such equation for special choices of $U$ is obtained for instance in \cite{BNP2011}.
 In contrast with finite strategy spaces, where the notion of equilibrium is well understood and studied \cite{HSbook,Weibull1995}, the situation of games with infinite strategies has been missing for a long time a general theory due to the technical and conceptual difficulties stemming from understanding which notion of distance between probability measures was the most suitable to use \cite{Norman2008}. 
 Our approach uses the classical transport distances, and the general frame of evolution problems in the class of probability measures, see for instance \cite{AGS2008} for a systematic treatment of this topic; see also  \cite{CCR,Piccoli2017} for recent contributions.
Some results of asymptotic behavior and the stability of solutions are given in \cite{BNP2011}.

\subsection{Spatially inhomogeneous replicator dynamics}

In this paper, differently from spatially homogenous dynamical games, we assume that the population of players is distributed over a {position} space and that they are each endowed with probability distributions of strategies, which they draw at random to evolve their {positions}. The positions of the players are assumed to be in the $d$-dimensional Euclidean space $\mathbb R^d$ and the pure strategies $u$ are in a compact metric space $U$; a probability measure $\sigma\in\Probabilities{U}$ denotes a mixed strategy.

With these definitions, the space of pairs of positions and mixed
strategies is $C\coloneqq\mathbb R^d \times \Probabilities{U}$, whose
elements are pairs ${\cbl{} y=}(x,\sigma)$ describing the state of a player. 
The system will be described by the evolution of
a measure $\Sigma\in\Probabilities{C}=\Probabilities{ \mathbb R^d
  \times\Probabilities{U}}$
\on our state space, which represents a distribution of players with strategies.
Notice that such a measure $\Sigma$ well describes the superposition of players with different strategies that at some time occupy the same 
position.

Omitting for the time being the temporal variable, we proceed to the description of the dynamics of the pair $y\coloneqq(x,\sigma)\in C$.
The player $x$ moves with a velocity which is obtained by averaging over all the strategies a suitable function $e\colon \mathbb{R}^d \times U\to\mathbb{R}^d$, namely
\begin{equation}\label{mm001}
\dot x=\int_U e(x,u)\,\de\sigma(u).
\end{equation}
For instance, in the simplest case when $e(x,u)=u$ and the set of strategies $U$ consists of a finite subset of
$\mathbb{R}^d$, this ``mean velocity'' dynamics can be thought of as the outcome of a faster time scale.
For convenience, it is useful to see the right-hand side of \eqref{mm001} as the result of a map $a\colon C\to \mathbb{R}^d$, 
by means of defining
\begin{equation}\label{mm010}
a(y)=a(x,\sigma)\coloneqq\int_U e(x,u)\,\de\sigma(u).
\end{equation}
Notice that $a$ depends linearly on $\sigma$.
In order to write the evolution law for the strategies 
we will assume that it is driven by an interaction mechanism, which
depends only on the state of the system (position and strategies) and does not distinguish two players 
occupying the same place with the same instantaneous strategy distribution.
Thus we consider a Lipschitz function 
\begin{equation}
\label{ourJ}
J\colon (\R^d\times U)^2\to\mathbb{R}
\end{equation}
and we define an interaction potential
\begin{equation}\label{mm002}
\Delta\colon\Probabilities{C}\times C \to \mathrm C(U), \qquad  (\Sigma,{\cbl{}y}) \mapsto \Delta_{\Sigma,{\cbl{}y}},
\end{equation}
by setting (notice that the integrals make sense, under suitable moment assumptions on $\Sigma$, since $J$ has at most
linear growth)
\begin{equation}\label{mm003}
\begin{split}
\Delta_{\Sigma,{\cbl{}y}}=\Delta_{\Sigma,{\cbl{}(x,\sigma)}}(u)\coloneqq &\int_C\int_U J(x,u,x',u')\,\de\sigma'(u')\,\de\Sigma(x',\sigma') \\
& -\int_U\int_C\int_U J(x,w,x',u')\,\de\sigma'(u')\,\de\Sigma(x',\sigma')\,\de\sigma(w). 
\end{split}
\end{equation}
The evolution law for the mixed strategies of the player at $x$
is again according to a replicator dynamics similar to the ones
mentioned above, and it can be written as 
\begin{equation}\label{mm004}
\dot\sigma=\Delta_{\Sigma,{\cbl{}(x,\sigma)}}\,\sigma.
\end{equation}
The interaction potential $\Delta$ has here the following simple interpretation: $J(x,u,x',u')$ represents the contribution to the payoff that the player $x$ gets from the pure strategy $u$ assuming that the player $x'$ acts with pure strategy $u'$.
When player $x'$ acts with mixed strategy $\sigma'$, we obtain
\begin{equation}\label{mm005}
\mathcal J(x,u,x',\sigma')\coloneqq\int_U J(x,u,x',u')\,\de\sigma'(u').
\end{equation}
Therefore, the full payoff of $x$ given the pure strategy $u$ is
\begin{equation}\label{mm005bis}
\int_C \mathcal J(x,u,x',\sigma')\,\de\Sigma(x',\sigma'),
\end{equation}
corresponding to the first integral in \eqref{mm003}.
The second term is the integral of this quantity with respect to $\sigma$, thus the payoff expected by $x$ from the mixed strategy
$\sigma$, given the full distribution $\Sigma$. 

We remark that the dependence of the payoff in \eqref{mm005} on the full strategy $\sigma'$ of $x'$, {\cbl{} which may imply a certain level of anticipation,} is made for the sake of generality: in practical situations $\mathcal
J$ may only depend on a marginal of $\sigma'$, and the strength of the interaction between players
may be influenced by their distance $|x-x'|$ through the function $J$.
From another point of view, this
full dependence (as the expectations of the players on the future in the mean field games theory, see \cite{caha17}
and the short discussion in Section~\ref{sec:compMFG}) could emerge
from a repetition of the evolutionary game, see \cite{BorgersSarin97} for a contribution in this direction in the spatially homogeneous
case.   

Putting together \eqref{mm010} and \eqref{mm004}, we can write the evolution for $y$ as
\begin{equation}\label{mm006}
\dot y=(\dot x,\dot\sigma)=\big(a(x,\sigma),\Delta_{\Sigma,{\cbl{}(x,\sigma)}}\,\sigma\big)\eqqcolon b_\Sigma(y),
\end{equation}
which can be interpreted as an ODE in the convex set $C$.
The theory of ODEs in Banach spaces (see Appendices \ref{app:calc} and \ref{sec:regularity}, or \cite[Partie II]{Cartan1967}, \cite[Chapter X]{Dieudonne1960} for classical monographs on this topic) will be a useful tool in the study of
the well posedness of \eqref{mm006}, the existence, uniqueness, and stability of its solutions, and their properties.

In order to do so, we embed $C$ in the Banach space
$Y\coloneqq\mathbb{R}^d\times F(U)$, where
$F(U)\coloneqq\overline{\operatorname{span}(\Probabilities{U})}^{\norm{\cdot}_{\rm
    BL}}$ and the closure is taken in  the dual space $(\mathrm{Lip}(U))'$ with respect to  
the dual norm (also called bounded Lipschitz norm)
\begin{equation}\label{mm200}
\norm\ell_{\rm BL}\coloneqq\sup \big\{\langle \ell,\varphi\rangle:  \varphi\in \mathrm{Lip}(U),\   \|\varphi\|_{\Lip}\le 1\big\},
\end{equation}
where $\|\cdot\|_{\mathrm{Lip}}$ is the Lipschitz norm in \eqref{eq:Lipnorm} below.
Notice that $F(U)\subset (\mathrm{Lip}(U))'$, and that $(Y,\norm\cdot_Y)$, with $\norm{y}_Y=\norm{(x,\sigma)}_Y\coloneqq|x|+\norm\sigma_{\rm BL}$, is a separable Banach space.
The space $F(U)$ defined above, known in the literature as the Arens-Eells space \cite{AE1956}, is isometric to the predual of ${\rm Lip}(U)$,  see \cite{AP,We} for more details also on the space $(\mathrm{Lip}(U))'$.

\subsection{Formal derivation of a nonlinear master equation}\label{formalder}\protect 

We fix a time interval $[0,T]$, a time step $h=T/N$ and an initial datum $\bar\Sigma\in \Probabilities{\mathbb R^d\times\Probabilities{U}}$, assuming for simplicity 
that the first marginal of $\bar\Sigma$ is compactly supported. Recalling the notation $C=\mathbb R^d\times\Probabilities{U}$ and its natural
structure of convex set, we build a discrete solution 
$$M_h\in\Probabilities{\rmC([0,T];C)}$$ 
concentrated on paths $(x(t),\sigma(t))\colon[0,T]\to C$, which are piecewise affine (in the $N$ intervals). 
We denote $\Sigma_{t,h}\coloneqq(\mathrm{ev}_t)_\#M_h$, where $\mathrm{ev}_t\colon \rmC([0,T];C) \to C$ 
defined by $\mathrm{ev}_t(x,\sigma)\coloneqq (x(t),\sigma(t))$. 
In particular, $(\mathrm{ev}_0)_\#M_h =\bar\Sigma$ is the given initial condition.

The heuristic idea is the following: If the player at time $t=ih$, for $i\in\{0,\ldots,N\}$, is in the
position $\bar x$, with mixed strategy represented by the probability measure
$\bar\sigma$, first they upgrade their belief on the probability replacing $\bar\sigma$ by
$$
\bar\sigma'\coloneqq(1+h\Delta_{\Sigma_{t,h},{\cbl{}(\bar x,\bar\sigma)}})\,\bar\sigma.
$$
Then, they move to the next position $\bar x+h e(\bar x,u)$ choosing $u$ with probability $\bar\sigma'$ 
and carrying the same probability $\bar\sigma'$  to the new position. 
The probability measure $M_h$ takes all the future stochastic realizations into account.
In more formal terms, the conditional probability relative to $M_h\vert_{[0,t+h]}$ of $(x(t+h),\sigma(t+h))$, given the information 
that at time $t$ one has $(x,\sigma)(t)=(\bar x,\bar\sigma)$, is $((\bar x+h e(\bar x,\cdot))_\#\bar\sigma')\times\delta_{\bar\sigma'}$. 
By iterating this process $N$ times one can build $M_h$ on the whole time interval $[0,T]$ (alternatively, one can view this as a 
Markov process with the above defined transition probabilities and build first a measure $N_h$ in $\Probabilities{C}^{N+1}$, then 
$M_h$ by associating to $N+1$ points in $\Probabilities{C}$ a piecewise affine path in $[0,T]$).

Under boundedness assumptions on the field $e\colon\mathbb R^d\times U\to \mathbb R^d$ and on the interaction potential $\Delta$, 
it turns out that $M_h$ is concentrated on paths $(x(t),\sigma(t))$ satisfying the equi-Lipschitz property
\begin{equation}\label{eqLip}
  |x(s)-x(t)|\leq L |s-t|,\quad\|\sigma(s)-\sigma(t)\|_{\rm BL}\leq L|s-t|, \qquad \text{for all $0\leq s\leq t\leq T$,}
\end{equation}
with $L=L(e,J,U)\geq 0$. 

Given now a bounded test function $\Phi\colon Y  \to \mathbb R$ of class $\rmC^1$ (in the Fr\'echet sense) with respect to the $\|\cdot\|_{Y}$ norm, 
let us write a discrete continuity equation associated to $\Sigma_{t,h}$, the marginals of $M_h$ in the sense 
$\Sigma_{t,h}\coloneqq(\mathrm{ev}_t)_\#M_h$.
For $t=ih$, $0\leq i\leq N-1$, one has 
\begin{eqnarray*}
&&\int\Phi(x,\sigma){\, \crd  \mathrm{d}}(\Sigma_{t+h,h}-\Sigma_{t,h})(x,\sigma)\\&=&\int\bigl[ \Phi(x(t+h),\sigma(t+h))-\Phi(x(t),\sigma(t))\bigr]\,{\crd  \mathrm{d}}M_h(x(\cdot),\sigma(\cdot))\\
&=&\int\bigl[ \Phi(x(t+h),(1+h \Delta_{\Sigma_{t,h},x(t),\sigma(t)})\sigma(t))-\Phi(x(t),\sigma(t))\bigr]\,{\crd  \mathrm{d}}M_h(x(\cdot),\sigma(\cdot))\\
&\sim& h \int D \Phi(x(t),\sigma(t))\cdot \big(a(x(t),\sigma(t)),\Delta_{\Sigma_{t,h},x(t),\sigma(t)}\sigma(t) \big)\,{\crd  \mathrm{d}}M_h(x(\cdot),\sigma(\cdot))\\
&=& h\int D \Phi(x,\sigma)\cdot b_{\Sigma_{t,h}}(x,\sigma)\,{\crd  \mathrm{d}}\Sigma_{t,h}(x,\sigma),
\end{eqnarray*}
where $b_\Sigma$ is given by \eqref{mm006}, and we used the chain rule for Fr\'echet differentiation.

Recalling from \eqref{eqLip} that $M_h$ is concentrated on equi-Lipschitz paths $t \mapsto (x(t), \sigma(t))$, the family 
$\{M_h\colon h=T/N,\,N\geq 1\}$ is weakly compact in the space  $\Probabilities{\rmC([0,T];C)}$ and by Prokhorov theorem  
it has limits points as $h\to 0$. Any limit point $M$ is concentrated on Lipschitz paths satisfying \eqref{eqLip} and
this construction builds a continuous map $\Sigma\colon[0,T]\to C$ by $\Sigma_{t}\coloneqq(\mathrm{ev}_t)_\#M$, satisfying the equation
\begin{equation}
\label{eq:master0}
\partial_t \Sigma_t + \operatorname{div} (b_{\Sigma_t}\,  \Sigma_t)=0,
\end{equation}
in the weak sense of \eqref{eq:32} below, with $\bar\Sigma$ as initial
condition. In the following, we explore notions of solutions to
\eqref{eq:master0} and conditions for their existence and uniqueness,
starting from the corresponding finite agents model. 

\subsection{Comparison with mean-field games}\label{sec:compMFG}

The master equation \eqref{eq:master0} is a novel model of spatially non-homogenous evolutive games, which fuses mean-field theory, optimal transport, and replicator dynamics from evolutionary game theory.
There are other approaches towards modeling spatially non-homogenous games of a large population of indistinguishable agents. Perhaps the most prominent is the so-called theory of {\it mean-field games}. This class of problems was considered in the economics literature by Boyan Jovanovic and Robert W.~Rosenthal \cite{joro88}, 
  in the engineering literature by Peter E.~Caines, Minyi Huang, and Roland P.~Malham\'e \cite{huang2005a,huang2006a}, 
  and independently and around the same time by the mathematicians Jean-Michel Lasry and Pierre-Louis Lions \cite{LL2007}.\\
 
In continuous time a mean-field game is typically composed of a Hamilton-Jacobi-Bellman equation for the optimal control problem of an individual, and a forward Fokker-Planck-Kolmogorov equation for the dynamics of the aggregate distribution of agents. 
Under fairly general assumptions, it can be proven that a mean-field game is the limit as $N\to \infty$ of a $N$-player Nash equilibrium.
In particular, one can consider the stochastic evolution of $N$ players dictated by equations
\begin{equation}\label{eq:sde}
\de X_t^i= u_t^i \de t + \sqrt {2\nu} \de B_t^i, \qquad i=1,\ldots,N,
  \end{equation}
 where $X_0^i$ are independently drawn at random according to a probability distribution $\mu_0 \in \Probabilities{\R^d}$, $B_t^i$ are independent $d$-dimensional Brownian motions, and $\nu>0$ is a noise parameter, so that the limiting case $\nu=0$ corresponds to the so-called first-order games. The player $i$ can choose a strategy $u^i$
adapted to the filtration
 \begin{equation}\label{eq:filtr}
\mathcal F_t= \sigma(X_0^j,B_s^j\colon s\leq t,\,j=1,\ldots,N).
\end{equation}
The payoff of player $i$ is given by 
  \begin{equation}\label{eq:payoff}
J_i^N(u^1,\dots,u^N)= \mathbb E \left [ \int_0^T \frac{1}{2} |u_t^i|^2  + F\bigg (X_t^i,\frac{1}{N-1} \sum_{j\neq i} \delta_{X_t^j} \bigg ) \de t    \right ].
\end{equation}
The solution of the $N$-player game is the suitable minimization of \eqref{eq:payoff} under the constraints \eqref{eq:sde} for all $i=1,\dots,N$. Hence, the appropriate notion of solution may be precisely the Nash equilibrium, i.e., a configuration of strategies $(u^{*,1},\dots,u^{*,N})$ such that
$$
J_i^N(u^{*,1},\dots,u^{*,N}) \leq J_i^N(u^{*,1},\dots,u^{*,i-1},u, u^{*,i+1}, \dots,u^{*,N}), \qquad \mbox{for all } u,
$$
for all $i=1,\ldots,N$. The computation of such equilibria becomes intractable already for a moderate number $N$ of players. However, for $N$ very large a mean-field approximation for $N\rightarrow \infty$ may help to obtain approximate Nash equilibria. In particular for $N\rightarrow \infty$ one can approximate the empirical distribution $\mu_t^N =\frac{1}{N} \sum_{i=1}^N \delta_{X_t^i}$ supported on realizations of \eqref{eq:sde} by the solution of the forward Fokker-Planck-Kolmogorov equation
\begin{equation}\label{FPeq}
\partial_t \mu - \nu\Delta \mu {\crd + \div (u \mu)} =0,
\end{equation}
for an appropriate control function $u(t,x)$. {\cbl{} For a given time-dependent distribution $\hat{\mu}_t$, the} 
corresponding payoff functional would be given by
\begin{equation}\label{OC}
J(u) =  \int_0^T  \int_{\R^d} \left ( \frac{1}{2} |u(t,x)|^2  + F\left (x,\hat \mu_t \right )\right) \de\mu_t \de t.
\end{equation}
The solution of the mean-field game comes from the minimization of \eqref{OC} under the PDE constraints \eqref{FPeq} and the fixed point condition $\mu_t= \hat \mu_t$, and these yield the mean-field game system
\begin{equation}\label{MFGeq}
\left \{
\begin{aligned}
  -\partial_t \psi -\nu\Delta \psi + \frac{1}{2} |\nabla \psi |^2&= F(x, \mu),\\
   \partial_t \mu - \nu\Delta \mu  -  \div (\nabla \psi \mu) &=0,
  \end{aligned}
   \right .
\end{equation}
with appropriate inital and terminal conditions. The choice {\crd  $u(t,x) =- \nabla \psi(t,x)$} 
yields  nearly optimal strategies to be inserted in \eqref{eq:sde} by posing $u_t^i= u(t,X_t^i)$.
Several comments about differences between our spatially non homogeneous evolutionary game and mean-field games are in order:
\begin{itemize}
\item Mean-field games are  intimately linked to the theory of (stochastic) mean-field control and as such are global-in-time optimizations: in other words, one has to forecast the future behavior of the agents in order to be able to solve the game and this is mathematically expressed by the need of simultaneously solving a backward-in-time evolution to compute the optimal strategy. 
Our model is a simple forward-in-time evolution according to the master equation \eqref{eq:master0} in a continuous local-in-time search for pairwise-game equilibria.
\item Mean-field games are built around the concept of Nash equilibrium for non-cooperative games, one of the notion of solutions considered in evolutionary game theory, which also aims at reaching different situations.
\item Well-posedness of  the mean-field game system \eqref{MFGeq} has been shown  for  special choices of costs $F$, 
the so-called potential games \cite{LL2007}. Accordingly, the numerical solution of the mean-field game system, see also, e.g., \cite{acca10} for alternative approaches, could be based on the iterative solution of backward-forward system by means of individual solvers for the two equations: starting from a given initial trajectory $t\mapsto \mu_t^{(0)}$, one iterates the numerical solution for $n=0,1,2,\ldots$
 \begin{equation}
\label{MFGeqnum} 
\left \{
\begin{array}{l}
  -\partial_t \psi^{(n)} -{\cbl{}\nu}\Delta \psi^{(n)} + \frac{1}{2} |\nabla \psi^{(n)} |^2= F(x,{\crd  \bar \mu^{(n)}}),\\
   \partial_t \mu^{(n+1)} - {\cbl{}\nu}\Delta \mu^{(n+1)} -  \div (\nabla \psi^{(n)} \mu^{(n+1)}) =0,
  \end{array}
   \right .
   \end{equation}
where $ \bar \mu^{(n)}= \frac{1}{n} \sum_{i=1}^n \mu^{(i)}$, and with appropriate initial and terminal conditions. This iterative procedure, also called  {\it learning}  or {\it fictitious play} in \cite{caha17}
{\cbl provides a justification to the ``knowledge about the future'' incorporated in the model}, and is shown to converge, at least for potential games. For our model,
we will prove that the well-posedness (existence, stability, and uniqueness) of the master equation \eqref{eq:master0} is ensured under Lipschitz assumptions on the function $e$ in \eqref{mm010} and on our pairwise-game payoff function $J$ defined in \eqref{ourJ}. In addition, the formal derivation of the master equation \eqref{eq:master0} as in Section~\ref{formalder} provides already a rather clear path towards a time marching numerical solution.
\end{itemize}

\section{Measure theoretic preliminaries}
Before exploring notions of solutions to  \eqref{eq:master0} and conditions for their existence and uniqueness, we make more precise the functional setting where the evolutions governed by \eqref{eq:master0} take place. We use here differential and Bochner calculus in separable Banach spaces and we refer to the 
Appendix~\ref{app:calc} for some related basic notions and results. 

\subsection{Notation and distances in the space of measures}

If $(X,\sfd_X)$ is a metric space, we denote by $\Measures(X)$ the space of signed Borel measures in $X$ with finite total variation,
by $\Measures_+(X)$, $\Probabilities{X}$ the convex subsets of nonnegative measures and probability measures respectively. 
For $\sigma\in\Measures(X)$, $|\sigma|\in\Measures_+(X)$ denotes the total variation measure of $\sigma$ (see also \eqref{eq:11} below).
{We shall also use the notation $\Measures_0(X)$ for the subset of measures with $0$ mean; we shall use the
identity $\R(\Probabilities{X}-\Probabilities{X})=\Measures_0(X)$ (in the sense of  Minkowski sums)
provided by the Hahn decomposition theorem for signed measures.

For a Lipschitz function $f\colon X\to\R$ we denote by
$$
\mathrm{Lip}(f)\coloneqq\sup_{x,\,y\in X \atop x\neq y}\frac{|f(x)-f(y)|}{\crd  \sfd_X(x,y)}
$$
the Lipschitz constant and denote by $\mathrm{Lip}_b(X)$ the space of bounded Lipschitz functions.

Given $\mu\in\Measures_+(X)$ and  $f {\crd \colon} X\to Y$, with $f$
$\mu$-measurable, we shall denote by $\PushForward{f}{\mu}\in\Measures_+(Y)$ the push-forward measure, having the same mass as $\mu$ and
defined by $\PushForward{f}{\mu}(B)=\mu(f^{-1}(B))$ for any Borel set $B\subset Y$ (when $\mu\in\Probabilities{X}$, in {\crd  probability theory}
it is also named law of $f$ under $\mu$); we shall also often use the change of variables formula
$$
\int_Y g\,\de\PushForward{f}{\mu}=\int_X g\circ f\,\de\mu
$$
whenever either one of the integrals makes sense.

In a complete and separable metric space $(X,\sfd_X)$,  we shall use the Kantorovich-Rubinstein
(possibly infinite) distance $W_1(\mu,\nu)$ in the class $\Probabilities{X}$; thanks to Kantorovich duality, the definition 
$$
W_1(\mu,\nu)\coloneqq\sup\left\{\int_X\varphi\,\de\mu-\int_X\varphi\,\de\nu\colon \text{$\varphi\in\mathrm{Lip}_b(X)$, $\mathrm{Lip}(\varphi)\leq 1$}\right\}
$$
is equivalent to the one
$$
W_1(\mu,\nu)\coloneqq\inf\left\{\int_{X\times X} {\crd  \sfd_X(x,y)}\,\de\Pi(x,y)\colon
\Pi(A\times X)=\mu(A),\,\,\Pi(X\times B)=\nu(B)\right\}
$$
involving couplings $\Pi$ of $\mu$ and $\nu$, but we shall mostly be working with the first one. Notice that $W_1(\mu,\nu)$ is finite
if $\mu,\,\nu$ belong to the space
\begin{equation}\label{perone}
\Probabilitiesone{X}\coloneqq \left\{\mu\in\Probabilities{X}\colon \text{$\int_X \sfd_X(x,\bar x)\,\de\mu(x)<+\infty$ for some $\bar x\in X$}\right\}
\end{equation}
and that $(\Probabilitiesone{X},W_1)$ is complete if $(X,\sfd_X)$ is complete.
Recall also that the convergence of $\mu_h\in\Probabilitiesone{X}$ to $\mu\in\Probabilitiesone{X}$ with respect to the distance $W_1$ is equivalent to weak convergence in the duality with bounded Lipschitz functions plus convergence of first
moments \cite{AGS2008}, namely $\int_X \sfd_X(\cdot, \bar x)\,\de\mu_h\rightarrow \int_X \sfd_X(\cdot,\bar x)\,\de\mu$
for all $\bar x\in X$.

We also need a $\rmC^1_b$ variant of Kantorovich duality, valid in separable Banach spaces, stated below.

\begin{lemma}[$\mathrm{C}^1$ duality]\label{sfdc1}
For any separable Banach space $Y$ one has that
$$
\sfd_{\rmC^1}(\mu,\nu) \coloneqq  \sup_{\phi \in \rmC^1_b(Y)\atop {\mathrm{Lip}(\phi)\leq 1}} 
\bigg ( \int_Y \phi \,\de\mu - \int_Y \phi \,\de\nu \bigg )
$$
defines a distance in $\Measures(Y)$, which coincides with $W_1$ when restricted to $\Probabilitiesone{Y}$. 
\end{lemma}
\begin{proof}
Symmetry and triangle inequality are obvious. To prove the non-degeneracy, notice that
any cylindrical function  $\phi(y)=\psi(\langle y,z_1'\rangle,\dots,\langle y,z_d'\rangle)$, with $\psi\colon\R^d\to\R$ bounded, continuously
differentiable, and Lipschitz and $z_j'\in Y'$, is a $\rmC^1_b$ function on $Y$
and, after suitable rescaling, satisfies $\mathrm{Lip}(\phi)\leq 1$.  If two measures 
$\mu,\,\nu$ have vanishing $\sfd_{\rmC^1}$-distance, then necessarily they must coincide
on the $\sigma$-algebra generated by such cylindrical functions.  From
\cite[Lemma at pag.~131-132]{Yosida80}, since $Y$ is separable there
exists a sequence $(z_i')_{i\in \N}$ in the unit ball of $Y'$ with 
\begin{equation}\label{goodyi}
\|y\|_Y=\sup_i\langle y,z_i'\rangle=
\sup_{i,k}\psi_k(\langle y,z_i'\rangle) ,\qquad\text{for all $y\in Y$,}
\end{equation}
where $\psi_k(r)=k\tanh(r/k)$.

Hence, balls and then open sets belong to the $\sigma$-algebra generated by cylindrical functions. 
It follows that this $\sigma$-algebra coincides with the Borel $\sigma$-algebra and $\mu=\nu$.

The proof of the final statement requires a slight refinement of the previous argument.
Let $\mu,\,\nu\in \Probabilitiesone Y$ and for every $a>0$ let us set
$\mathcal L_a\coloneqq
\big\{ \phi\in \Lip_b(Y):\sup_Y|\phi|\le a,\Lip(\phi)\le 1 \big\}$.
We consider the subset $\mathcal G_a$ of functions 
$\phi\in \mathcal L_a$ such that
\begin{displaymath}
  \int_Y \phi \,\de\mu - \int_Y \phi \,\de\nu\le \sfd_{\rmC^1}(\mu,\nu).
\end{displaymath}
By definition of $\sfd_{\rmC^1}$, 
$\mathcal G_a$ contains $\mathcal L_a\cap \rmC^1_b(Y)$;
we want to show that $\mathcal G_a$ coincides with $\mathcal L_a$.
Since $\mathcal G_a$ is closed with respect to pointwise convergence, it is
sufficient to prove that $\tilde {\mathcal G}_a\coloneqq\overline{\mathcal L_a\cap \rmC^1_b(Y)}$ coincides with $\mathcal L_a$. Notice that the
topology of pointwise convergence in $\mathcal L_a$ is metrizable: it
is sufficient to select a countable dense subset $(y_i)_{i\in \N}$ of
$Y$
and consider the distance 
\begin{displaymath}
  \sfd_{\mathcal L_a}(\phi,\psi)\coloneqq\sum_{i=0}^\infty 2^{-i}|\phi(y_i)-\psi(y_i)|.
\end{displaymath}
By approximating the convex function $S_N(r)\coloneqq r_1\lor r_2\lor
\cdots\lor r_N$, $r=(r_1,\ldots,r_N)\in \R^N$ by convolution with
a symmetric, nonnegative mollifier $\kappa\in C^\infty_c(\R^N)$, $\int \kappa=1$,
\begin{displaymath}
  S_{N,h}(r)\coloneqq h^N\int_{\R^N}S_N(r-s)\kappa(h s)\,\de s,\quad h\in \N,
\end{displaymath}
we obtain an increasing sequence of smooth functions converging to
$S_N$ as $h\to\infty$, and satisfying
\begin{displaymath}
  I_N(r)\le S_{N,h}(r)\le S_N(r),\qquad
  |S_{N,h}(r)-S_{N,h}(r')|\le \sup_{1\le n\le N}|r_n-r_n'|,
\end{displaymath}
where $I_N(r)\coloneqq -S_N(-r)=r_1\land r_2\land\cdots\land r_N$.
It follows that if $\psi_1,\ldots,\psi_N\in \tilde {\mathcal G}_a$
then $\psi_h\coloneqq S_{N,h}(\psi_1,\ldots,\psi_N)$ belongs to 
$\tilde {\mathcal G}_a$ and therefore also
$\psi=\psi_1\lor\psi_2\lor\cdots\lor\psi_N$ belongs to $\tilde
{\mathcal G}_a$. The same property holds for the infimum and 
extends to a countable family of functions.

By using the representation \eqref{goodyi} it is then easy to check
that any functions of the form $y\mapsto 
-a\lor ( b+\|y-y_0\|)\land a$ belongs to $\tilde{\mathcal G}_a$.
Since every $f\in \mathcal L_a$ can be expressed as 
$$
f(y)=\inf_{y'\in Y_0} -a\lor (f(y')+\|y-y'\|)\land a  \qquad\forall y\in Y
$$ 
for any dense and countable subset $Y_0$ of $Y$, we conclude.
\end{proof}

From now on we fix a compact metric space $U$ with a distance $\sfd$. 
The space $\Measures(U)$, when endowed with the total variation norm
$$\norm\sigma_{\rm TV}\coloneqq\sup\bigg\{\int_U \varphi\,\de\sigma : \varphi\in
\mathrm C(U), |\varphi|\leq1\bigg\}$$
has the structure of Banach space, isometrically isomorphic to the dual of $\mathrm C(U)$. We will also
use the representation formulas
  \begin{equation}
    \label{eq:11}
    \norm\sigma_{\rm TV}=|\sigma|(U)=\sup\bigg\{\int_U \varphi\,\de\sigma : \varphi\in
    \mathrm B_b(U), |\varphi|\leq1\bigg\},
  \end{equation}
where $\mathrm B_b(U)$ denotes the class of bounded Borel functions $\varphi\colon U\to\mathbb R$.

Mixed strategies can be interpreted as Borel probability measures in $U$.
The set of Borel probability measures $\Probabilities U$ is a convex and weakly$^*$ closed 
subset of $\Measures(U)$, corresponding to the class of nondecreasing linear functionals
$L\colon\mathrm C(U)\to\mathbb R$ with $L(1)=1$ in the dual representation of $\Measures(U)$. On the other hand,
$\Probabilities U$ can also be realized as a compact convex set of another Banach space, the dual $(\mathrm{Lip}(U))'$ of $\mathrm{Lip}(U)$: 
the norm in $\mathrm{Lip}(U)$ is given by
\begin{equation}\label{eq:Lipnorm}
\|\varphi\|_{\mathrm{Lip}}\coloneqq
\sup_{u\in U}|\varphi(u)|+\mathrm{Lip}(\varphi)
\end{equation}
so that, for $\ell\in (\mathrm{Lip}(U))'$,
\begin{equation*}
  \|\ell\|_{\rm BL}\coloneqq\sup \big\{\langle \ell,\varphi\rangle:
  \varphi\in \mathrm{Lip}(U),\ 
  \|\varphi\|_{\Lip}\le 1\big\}.
\end{equation*}
For later use, let us record the property
\begin{equation}\label{mm200bis}
\|\varphi\psi\|_{\mathrm{Lip}}\leq \|\varphi\|_{\mathrm{Lip}}\|\psi\|_{\mathrm{Lip}},
\end{equation}
{\crd  for all $\varphi, \psi \in \mathrm{Lip}(U)$.}

Actually, while the linear structure we need forces us to go beyond $\Probabilities{U}$, for our purposes it will be more convenient to work 
in a closed subspace of $(\mathrm{Lip}(U))'$, namely
$$F(U)\coloneqq\overline{\operatorname{span}(\Probabilities{U})}^{\norm{\cdot}_{\rm BL}}\subset (\mathrm{Lip}(U))'.$$
This space, also called Arens-Eells space in the literature, is a separable Banach space containing 
$\Measures(U)$. 

Notice that, for a measure $\nu\in \Measures_0(U)$, the BL norm is equivalent to the norm induced by the dual 
formulation of the $1$-Wasserstein distance: in fact 
for every $1$-Lipschitz function $\varphi\colon U\to \mathbb R$ and for every $x_0\in U$ we have
\begin{displaymath}
  \int_U \varphi\,\de\nu= \int_U  (\varphi(x)-\varphi(x_0))\,\de\nu(x) \le \|\nu\|_{\rm BL}\Big(\sup_{x\in U}\sfd(x,x_0)+1\Big),
\end{displaymath}
so that, with $D_U\coloneqq\min\limits_{x_0\in U}\max\limits_{x_1\in U}\sfd(x_0,x_1)\le \mathrm{diam}(U)$, one has
$$ 
   \|\nu\|_{\rm BL}\le
    \sup\bigg\{\int_U \varphi\,\de \nu:
    \mathrm{Lip}(\varphi)\le 1\bigg\}\le 
    (1+D_U)\|\nu\|_{\rm BL}.
$$
In particular, when $\nu=\mu_1-\mu_2$ with $\mu_1,\,\mu_2\in \Probabilities{U}$, Kantorovich duality gives
\begin{equation}\label{eq:4}
    \|\mu_1-\mu_2\|_{\rm BL}\le
    W_1(\mu_1,\mu_2)\le (1+D_U)\|\mu_1-\mu_2\|_{\rm BL}.
\end{equation}

We summarize the previous discussion in the following list of properties:
\begin{enumerate}
\item $(\Measures(U),\|\cdot\|_{\rm TV})$ is a Banach space
continuously imbedded in $(F(U),\|\cdot\|_{\rm BL})$ with
\begin{equation}
  \label{eq:3}
  \|\nu\|_{\rm BL}\le \|\nu\|_{\rm TV}{\crd ,}\qquad
  \text{for every }\nu\in \Measures(U).
\end{equation}
\item $\Probabilities U$ is a weakly$^*$ and 
  closed convex set of $\Measures(U)$, endowed with the total variation norm;
  it is also compact in $F(U)$. Thanks to \eqref{eq:4} and to the fact that $W_1$ metrizes the weak$^*$ convergence in
  $\Probabilities U$, the BL norm induces the weak$^*$ topology in $\Probabilities U$.
\item For every $\mu_1,\mu_2\in \Probabilities U$, besides \eqref{eq:4} and \eqref{eq:3} with $\nu=\mu_1-\mu_2$, we have
  \begin{equation}\label{pv520}
    W_1(\mu_1,\mu_2) \le D_U\|\mu_1-\mu_2\|_{\rm TV}.
   \end{equation}
\end{enumerate}

\subsection{Differentiable curves in the space of measures}\label{sec:diffcurves}

Let us now consider two curves $t\mapsto \sigma_t\in \Probabilities U$,
and $t\mapsto \nu_t\in F(U)$, $t\in [0,T]$.
We will assume that $\sigma$ and $\nu$ are continuous with respect to the BL norm
and we want to give a meaning to the differential equation
\begin{equation}
  \label{eq:1}
  \frac\de{\de t}\sigma_t=\nu_t,\qquad t\in [0,T].
\end{equation}
It is easy to check that the classical formulation of \eqref{eq:1} as an ODE in the Banach space 
$F(U)$ is equivalent to the weak formulation of \eqref{eq:1}, that reads as
\begin{equation}
  \label{eq:2}
  \int_U\varphi\,\de\sigma_t-
  \int_U\varphi\,\de\sigma_s=
  \int_s^t \langle \nu_\tau,\varphi\rangle\,\de\tau,\qquad
  \text{for every $\varphi\in \mathrm{Lip}(U)$, $s,\,t\in [0,T]$.}
\end{equation}
Indeed, since $\nu$ is continuous, the map
\begin{equation}\label{eq:5}
  t\mapsto N_t=\int_0^t \nu_\tau\,\de\tau
\end{equation}
is of class $\mathrm C^1$ and its derivative exists in the classical sense
\begin{equation}
  \label{eq:6}
  \lim_{h\to0}\bigg\|\frac1h(N_{t+h}-N_t)-\nu_t\bigg\|_{\rm BL}=0,\qquad\text{for
    every }t\in [0,T].
\end{equation}
From \eqref{eq:2} it follows that 
$$\int_U \varphi\,\de\sigma_t=\int_U\varphi\,\de\sigma_0+
\langle N_t,\varphi\rangle, \qquad\text{for every }\varphi\in \mathrm{Lip}(U)$$
and therefore the density of $\mathrm{Lip}(U)$ in $\mathrm C(U)$ implies $\sigma_t=\sigma_0+N_t$ for al $t\in [0,T]$. 
In particular \eqref{eq:6} gives that $t\mapsto\sigma_t\in F(U)$ is of class $\rmC^1$ 
and that \eqref{eq:1} holds in the classical sense.
\begin{remark}[Vector integral]
  Since
  $\nu$ is continuous,
  the integral in \eqref{eq:5} can be equivalently defined as a Cauchy-Riemann integral
  or as a Bochner integral (see Section~\ref{BI}).
\end{remark}
Let us now suppose that $\nu$ takes its values in the smaller space $\Measures (U)\subset F(U)$ and that
the stronger condition
\begin{displaymath}
  \int_0^T \|\nu_t\|_{\rm TV}\,\de t<+\infty
\end{displaymath}
is satisfied.
Notice that the TV norm is lower semicontinuous with respect to the BL
topology, so that the map $t\mapsto \|\nu_t\|_{\rm TV}$ is lower semicontinuous and
therefore Borel. From the representation formula of $\sigma_t$ it follows that
\begin{equation}
  \label{eq:7}
  \|\sigma_t-\sigma_s\|_{\rm TV}=\|N_t-N_s\|_{\rm TV}\le 
  \int_s^t \|\nu_\tau\|_{\rm TV}\,\de\tau, \qquad\text{for all $0\leq s\leq t\leq T$,}
\end{equation}
so that $t\mapsto \sigma_t\in\Measures(U)$ is absolutely continuous. In particular, if
$\|\nu_t\|_{\rm TV}\in L^\infty(0,T)$, then $\sigma_t$ is a Lipschitz curve. 
Eventually, if $\nu$ is even continuous with respect to the total variation norm, then $t\mapsto \sigma_t\in\Measures(U)$ is of class $\rmC^1$, namely
\begin{equation}
  \label{eq:9}
  \lim_{h\to 0}\bigg\|\frac1h(\sigma_{t+h}-\sigma_t)-\nu_t\bigg\|_{TV}=0,\qquad
  \text{for every $t\in [0,T]$.}
\end{equation}

\subsection{ODE's in the space of measures}
Let us now consider the case when the right hand side $\nu_t$ in \eqref{eq:1} is given by a time dependent family of operators 
$A(t,\sigma)\colon[0,T]\times\Probabilities U\to F(U)$, not necessarily linear with respect to $\sigma$; we assume that $A$ is a continuous map 
when both $\Probabilities U$ and $F(U)$ are endowed with the BL topologies.
Then, we are considering the ODE 
\begin{equation}
  \label{eq:12}
  \frac\de{\de t}\sigma_t=A(t,\sigma_t),\qquad t\in [0,T],
\end{equation}
in $F(U)$. Initially, one can also look for solutions $\sigma\in \rmC([0,T];
(\Probabilities U,\|\cdot\|_{\rm BL}))$ in the weak sense of \eqref{eq:2}, namely
\begin{equation}
  \label{eq:13}
  \int_U\varphi\,\de\sigma_t-
  \int_U\varphi\,\de\sigma_s=
  \int_s^t \langle A(\tau,\sigma_\tau),\varphi\rangle\,\de\tau,\qquad
  \text{for every $\varphi\in \mathrm{Lip}(U), 0\leq s\leq t\leq T$.}
\end{equation}
But, since $A$ is continuous, we deduce from the previous discussion that 
$t\mapsto\sigma_t\in F(U)$ is of class $\rmC^1$, therefore
\eqref{eq:12} holds
in the classical pointwise sense at every $t\in [0,T]$:
\begin{equation}\label{eq:14}
  \lim_{h\to
    0}\bigg\|\frac1h\big(\sigma_{t+h}-\sigma_t\big)-A(t,\sigma_t)\bigg\|_{\rm BL}=0,\qquad
  \text{for all $t\in [0,T]$.}
\end{equation}

If moreover $A$ maps $[0,T]\times \Probabilities U$ to the smaller space
$\Measures(U)$ and we know that
\begin{displaymath}
  \sup_{t\in [0,T]}\|A(t,\sigma_t)\|_{\rm TV}<+\infty
\end{displaymath}
then we deduce that the curve $t\mapsto \sigma_t$ is Lipschitz with
respect to the total variation norm. 
Finally, if $A\colon[0,T]\times \Probabilities U\to\Measures(U)$ is continuous when the target space
is endowed with respect to the total variation norm,
\eqref{eq:14} improves to 
\begin{equation}\label{eq:14bis}
  \lim_{h\to
    0}\bigg\|\frac1h\big(\sigma_{t+h}-\sigma_t\big)-A(t,\sigma_t)\bigg\|_{\rm TV}=0,\qquad
  \text{for all $t\in[0,T]$.}
\end{equation}

\section{Finite agent model}
\subsection{$1$-average-player evolution}

In order to explain how to apply the previous discussion to our model,
let us first consider the simplest, and somehow degenerate, 
case of a single player. Recalling the formal derivation of our model in Section~\ref{formalder},
this case should be considered as a limit of an evolution
process, when all the players (whose total mass is conventionally
normalized to $1$) are initially concentrated in the same initial
place $\bar x$ with the same initial distribution $\bar \sigma$ of strategies.

The evolution is then given by a moving point $y_t=(x_t,\sigma_t)\in \R^d\times \Probabilities U$ and 
$\Sigma_t$ is just the Dirac mass concentrated at $y_t$. Then $y$ satisfies the differential equation
\begin{equation}
  \label{eq:19}
  \left\{
    \begin{aligned}
      \dot x_t&=a(x_t,\sigma_t)=\int_U e(x_t,u)\,\de\sigma_t(u),\\
      \dot \sigma_t&=\bigg(\int_U J(x_t,\cdot,x_t,u')\,\de\sigma_t(u')-
      \int_{U\times U} J(x_t,w,x_t,u')\,\de\sigma_t(u')\,\de\sigma_t(w)\bigg)\sigma_t,
    \end{aligned}
    \right.
\end{equation}
where we denoted by $e\colon \mathbb{R}^d \times U\to \R^d$ the velocity field driving
the motion of the player according to the choice of a strategy $u$. 
This can be interpreted as a differential equation in the phase space
 $C=\R^d\times \Probabilities U\subset Y=\R^d\times F(U)$ of the form
\begin{equation}
  \label{eq:20}
  \dot y_t=b(y_t)
\end{equation}
where $b$ is the vector field
\begin{displaymath}
  b(x,\sigma)=\bigg(\int_U e(x,u)\,\de\sigma(u),
  \Big(\int_U J(x,\cdot,x,u')\,\de\sigma(u')-\int_{U\times U}
  J(x,w,x,u')\,\de\sigma(u')\,\de\sigma(w)\Big)\sigma\bigg).
\end{displaymath}
The $1$-average-player evolution thus reproduces the mechanism of a replicator equation
  influenced by a vector parameter $x$, whose dynamics is in turn
  affected by the evolving strategy distribution.
As we will see in the next more general cases, the particular structure of the vector field $b$ satisfies the structural assumptions 
of the main existence Theorem~\ref{thm:Brezis} and the second component $\sigma_t$ of the curve $y_t$
will be differentiable even with respect to the total variation norm.

\subsection{$N$-{\crd  average-}player system}
In the case of $N$ players, we have to follow the evolution of $N$
points $y_{i,t}=(x_{i,t},\sigma_{i,t})$, $i=1,\ldots, N$.
It is useful to introduce the interaction field $f\colon C\times C\to Y$ 
between two players (as usual, a player is identified by the position $x$ and a
mixed strategy $\sigma$): 
we can write $f$ as a pair $(f_x,f_\sigma)$ where $f_x\colon C\times C\to \R^d$ and
 $f_\sigma\colon C\times C\to F(U)$;
the first component is in fact independent of the interaction and can
be written as
\begin{equation}
  \label{eq:21}
  f_x(y,y')=f_x(x,\sigma,x',\sigma')=a(x,\sigma)=\int_U e(x,u)\,\de\sigma(u),
\end{equation}
whereas the second component $f_\sigma$ is given by
\begin{equation}\label{eq:21bis}
\begin{split}
  f_\sigma(y,y')= & f_\sigma(x,\sigma,x',\sigma') \\
  \coloneqq &
  \Big(\int_U J(x,\cdot,x',u')\,\de\sigma'(u')-\int_{U\times U} J(x,w,x',u')\,\de\sigma'(u')\de\sigma(w)\Big)\sigma.
  \end{split}
\end{equation}
The system
\begin{displaymath}
      \left\{
    \begin{aligned}
      \dot x_{i,t}&=a(x_{i,t},\sigma_{i,t})=\int_U e(x_{i,t},u)\,\de\sigma_{i,t}(u),\\
      \dot \sigma_{i,t}&=
      \bigg(\frac1N\sum_{j=1}^N
      \Big(\int _U J(x_{i,t},\cdot,x_{j,t},u')\,\de\sigma_{j,t}(u')-
      \int_U
      J(x_{i,t},w,x_{j,t},u')\,\de\sigma_{j,t}(u')\,\de\sigma_{i,t}(w)\Big)\bigg)
      \sigma_{i,t},
    \end{aligned}
    \right.
\end{displaymath}
can be rewritten in the compact form
\begin{equation}
  \label{eq:23}
  \dot y_{i,t}=\frac 1N\sum_{j=1}^N f(y_{i,t},y_{j,t}).
\end{equation}
We will see in the next section that such a problem always admits a unique solution, whenever $f$ satisfies the following two conditions:
\begin{enumerate}
\item $f$ is Lipschitz from $C\times C$ to $Y$ when $Y$ is endowed with the product norm 
induced by the Euclidean norm in $\R^d$ and the BL norm in $F(U)$
and the convex subset $C$ is endowed with the distance induced by the inclusion in $\R^d\times F(U)$;
\item for every $R>0$ there exists a constant $\theta>0$ such that 
  for every $x,\,x'\in B_R(0)\subset \R^d$ and $\sigma,\sigma'\in \Probabilities U$ 
  \begin{equation}
    \label{eq:24}
    \sigma+\theta f_\sigma(x,\sigma,x',\sigma')\in \Probabilities U.
  \end{equation}
This last condition is needed to have condition \eqref{eq:16} in Theorem \ref{thm:Brezis} fulfilled.
\end{enumerate}
The above conditions are surely satisfied if the function $J$ is
Lipschitz in $(\R^d\times U)^2$.

Notice that \eqref{eq:23} admits an equivalent formulation by
introducing the time-dependent distribution $\Sigma_{t}^N\coloneqq\frac1N\sum_{i=1}^N\delta_{y_{i,t}}\in \Probabilities X$
and the associated vector field
\begin{equation}
  \label{eq:25}
  b_{\Sigma^N_t}(y)=b_{\Sigma^N}(t,y)\coloneqq\int_Y f(y,y')\,\de\Sigma_t^N(y') =\frac 1N\sum_{i=1}^N\sum_{i=1}^Nf(y,y_{i,t}).
\end{equation}
Such a vector field induces a family $\bY_{\Sigma^N}(t,s,\cdot)\colon C\to C$, for $0\le s<t\le T$, of transition operators (also called flow map) associated to the initial value problem
\begin{equation}
  \label{eq:26}
  \dot y_r=b_{\Sigma^N}(r,y_r),\qquad
  y_s=y,
\end{equation}
namely $\bY_{\Sigma^N}(t,s,y)\coloneqq y_s$.

Therefore \eqref{eq:23} reads as
\begin{equation}
  \label{eq:27}
  y_{i,t}=\bY_{\Sigma^N}(t,s,y_{i,s}),\quad
  \text{so that}\;
  \Sigma_t^N=\bY_{\Sigma^N}(t,s,\cdot)_\#\Sigma_s^N,\qquad
  \text{for every $0\le s\leq t\le T$.}
\end{equation}

We conclude this section by further elaborating the expressions above. By a suitable exchange of integrals, 
we rewrite the following equation with the equivalent notation
\begin{equation*}
\begin{split}
 \dot \sigma_{i,t}=&
      \bigg(\frac1N\sum_{j=1}^N
      \Big(\int_U J(x_{i,t},\cdot,x_{j,t},u')\, {\crd  \mathrm{d}} \sigma_{j,t}(u')-
      \int_{U \times U}
      J(x_{i,t},w,x_{j,t},u')\,\de\sigma_{j,t}(u')\de\sigma_{i,t}(w)\Big)\bigg)
      \sigma_{i,t},\\
      =& 
      \bigg(\int_C
      \Big(\int_{U} J(x_{i,t},\cdot,x',u')\, {\crd  \mathrm{d}}\sigma'(u')-
      \int_{U \times U}
      J(x_{i,t},w,x',u')\, \de\sigma'(u') \de\sigma_{i,t}(w)\Big) {\crd  \mathrm{d}} \Sigma_t^N(x',\sigma')\bigg)
      \sigma_{i,t}  \\
      =&\bigg(  J \ast  \Sigma_t^N(x_{i,t}, \cdot) - \int_U  J\ast \Sigma_t^N(x_{i,t}, w)\,\de\sigma_{i,t}(w) \bigg) \sigma_{i,t},
\end{split}
\end{equation*}
where in the sequel we shall use the compact notation
\begin{equation}\label{astnotation}
J\ast \Lambda(x, u) \coloneqq \int_C \int_U J(x,u,x',u') \,\de\sigma'(u')\de \Lambda(x',\sigma')
\end{equation}
for $\Lambda\in\Probabilities{C}$.

\subsection{Distributed players system: Eulerian and Lagrangian solutions}
The general problem associated to an arbitrary initial distribution of 
players $\bar\Sigma\in \Probabilities C$ 
can be described as follows. Recall that we are endowing $C$ with the distance 
\begin{equation}
  \label{eq:29}
  \sfd_C(y_1,y_2)\coloneqq |x_1-x_2|+\|\sigma_1-\sigma_2\|_{\rm BL}, \qquad
  y_i=(x_i,\sigma_i).
\end{equation}

First of all, in order to take care of the lack of compactness of $\R^d$, the first factor of $C$, 
we assume that the first moment of the first marginal of $\bar\Sigma$ is finite:
\begin{equation}
  \label{eq:22}
  \int _C|x|\,\de\bar\Sigma(x,\sigma)=M_0<+\infty.
\end{equation}
Since $U$ is compact, \eqref{eq:22} holds if and only if $\bar\Sigma\in\Probabilitiesone C$, with
$\Probabilitiesone C$ defined as in \eqref{perone}.

We observe that for every continuous curve 
$t\in [0,T]\mapsto \Sigma_t\in\Probabilitiesone C$ it is possible to define
a time dependent vector field
\begin{equation}\label{pv2}
b_\Sigma(t,y)=b_{\Sigma_t}(y)=\int_C f(y,y')\,\de\Sigma_t(y')
\end{equation}
with {\crd  $f=(f_x,f_\sigma)$} as in \eqref{eq:21}, \eqref{eq:21bis}.
Since $f(y,\cdot)$ is continuous with linear growth and $\Sigma_t\in\Probabilitiesone C$,
the integral above can be interpreted as a Bochner integral, see Section~\ref{BI}.

We can then associate to $b_\Sigma$ the transition maps
$\bY_\Sigma(t,s,y)$ induced by ODE in $Y$
\begin{equation}
  \label{eq:30}
  \dot y_t=b_{\Sigma_t}(y_t),\qquad y_s=y.
\end{equation}
A solution $y=(x,\sigma)$ to \eqref{eq:30} satisfies, with the notation \eqref{astnotation},
\begin{displaymath}
      \left\{
    \begin{aligned}
      \dot x_{t}&=a(x_t,\sigma_t)=\int_U e(x_t,u)\,\de\sigma_{t}(u),\\
      \dot \sigma_{t}&=
      \bigg(\int_C
      \Big(\int_U J(x_{t},\cdot,x',u')\,\de\sigma'(u')-
      \int_{U\times U} J(x_t,w,x',u')\, \de \sigma'(u') \de\sigma_{t}(w)\Big)\,\de \Sigma_t(x',\sigma')\bigg) \sigma_{t},\\
      &= \bigg( J\ast\Sigma_t(x_t,\cdot) - \int_U J\ast \Sigma_t(x_t,w)\,  \de \sigma_t(w)\bigg) \sigma_{t}
    \end{aligned}
    \right.
\end{displaymath}
and the existence of a solution to \eqref{eq:30}  follows again by
Theorem~\ref{thm:Brezis}, see Theorem \ref{prob:mainBanach} under the structural properties assumed in Section \ref{sec:structprop} below.

Whenever we have at our disposal the flow map $\bY_\Sigma$, the
transported measures $\hat\Sigma_t\coloneqq \bY(t,0,\cdot)_\#\bar\Sigma$ solve an infinite-dimensional continuity equation
driven by the vector field $b_{\Sigma_t}$ 
given by \eqref{pv2}, namely (in integral form)
\begin{equation}\label{eq:32}
  \begin{aligned}
    \int_C \phi(t,y)\,\de\hat\Sigma_t(y)&- \int_C
    \phi(0,y)\,\de\hat \Sigma_0(y)\\&= \int_0^t \int_C
    \Big(\partial_s \phi(s,y)+ \int_C\mathrm D\phi(s,y)(f(y,y'))
    \,\de\Sigma_s(y')\Big)\,\de \hat\Sigma_s(y)\de s
  \end{aligned}
\end{equation}
for every $\phi\in\rmC_b^1([0,T]\times Y)$. Indeed, using the change of variables formula
for the push-forward measure, the chain rule, and once more the change of variables formula, one has
\begin{equation*}
\begin{split}
  \frac\de{\de t}\int_C \phi(t,y)\,\de
  \hat\Sigma_t(y)=&\int_C \frac\de{\de t}\phi\bigl(t,\bY_\Sigma(t,0,z)\bigr)\,\de\bar\Sigma (z)\\
  =&\int_C \Big(\partial_t \phi(t,\bY_\Sigma(t,0,z))+
  \mathrm D\phi(t,\bY_\Sigma(t,0,z))(b_{\Sigma_t}(\bY_\Sigma(t,0,z)))\Big)\,\de\bar\Sigma(z)\\ 
  =&\int_C \Big(\partial_t \phi(t,y)+ \mathrm
  D\phi(t,y)(b_{\Sigma_t}(y))\Big)\,\de\hat \Sigma_t(y)\\
=&  \int_C \Big(\partial_t \phi(t,y)+\int_C\mathrm D\phi(t,y)(f(y,y'))
\,\de\Sigma_t(y')\Big)\,\de\Sigma_t(y).
\end{split}
\end{equation*}
Formula \eqref{eq:32} follows now by integration in time.

We look for an evolving distribution $\Sigma$ which is self-transported by the generated vector field
$b_\Sigma$, so that $\hat \Sigma=\Sigma$. These facts motivate the following definition. 

\begin{definition}[Lagrangian and Eulerian solutions]
  \label{def:main}
  Let $\Sigma\in \mathrm{C}^0([0,T];(\Probabilitiesone{C},W_1))$ and $\bar\Sigma\in \Probabilitiesone C$.
  We say that $\Sigma$ is  an Eulerian solution of the initial value problem for the master equation starting from
  $\bar\Sigma$ if 
  $\Sigma_0=\bar\Sigma$ and \eqref{eq:32} holds with $\hat \Sigma=\Sigma$.
    We say that $\Sigma$ is a Lagrangian solution starting from $\bar\Sigma$ if
\begin{equation}\label{pv4}
  \Sigma_t=\PushForward{\bY_\Sigma(t,0,\cdot)}{\bar\Sigma}
  \qquad\text{for every }0\le t\le T,
\end{equation}
where $\bY_\Sigma(t,s,y)$ are the transition maps associated to the ODE \eqref{eq:30}.
\end{definition}

The Lagrangian notion of solution given by the transport identity \eqref{pv4} is rather standard and accepted in the literature 
of multi-agent systems and mean-field equations.  One can see, for instance, the notion of solution given in \cite[Definition 3.3]{CCR}. 
On the other hand, as in fluid mechanics, when looking at the evolution of spatially averaged quantities it is also important to derive an 
alternative Eulerian description in terms of a PDE, in our case \eqref{eq:32}.

We shall first address the problem of existence and uniqueness of Lagrangian solutions. Given that, as we illustrated above,
Lagrangian solutions are Eulerian, this settles the existence problem also for Eulerian solutions. The uniqueness of Eulerian
solutions is technically harder, and it will be dealt with in Section~\ref{sec:euuni}.

\begin{theorem}
  \label{thm:main}
Suppose that $J\colon(\R^d\times U)^2\to \R$ and $e\colon\R^d\times U\to\R$ are Lipschitz maps and let 
$f\colon C\times C\to Y$ be defined as in \eqref{eq:21}, \eqref{eq:21bis}.
Then, for every $\bar\Sigma\in \Probabilitiesone C$, there exists a
unique Lagrangian solution $\Sigma$ and its flow
map $\bY_\Sigma(t,s,y)=\bigl(\xx(t,s,y),\ssigma(t,s,y)\bigr)$ satisfies the additional regularity property that
$\ssigma(\cdot,s,y)$ is of class $\rmC^1$ with values in $(\Measures(U),\|\cdot\|_{TV})$, with
\begin{equation}\label{eq:additionalTV}
{\rm Lip}\bigl(\ssigma(\cdot,s,y),(\Measures(U),\|\cdot\|_{TV})\bigr)\leq L_J\mathrm{diam}(U).
\end{equation}

Moreover, there exists $L\geq 0$ such that for every pair $\bar\Sigma^i$, $i=1,\,2$, of initial data in $\Probabilitiesone C$, 
the corresponding solutions $\Sigma^i_t$ satisfy 
\begin{equation}\label{eq:33}
W_1(\Sigma_t^1,\Sigma_t^2)\le \mathrm e^{Lt}\,W_1(\bar\Sigma^1,\bar\Sigma^2),\qquad \text{for every $t\in [0,T]$.}
\end{equation}
\end{theorem}

The proof of Theorem~\ref{thm:main} will be given in Section \ref{sec:existuniq}: it does not depend on the particular structure of $f$ and $J$, but relies
on {\crd  their Lipschitz property, the convexity of $C$, and} the Banach framework. Notice that \eqref{eq:additionalTV} comes immediately
from \eqref{eq:21bis} and the definition of the $F(U)$-component $(b_{\Sigma_t})_\sigma(y)=\int_Y f_\sigma(y,y')\,\de\Sigma_t(y')$ of $b_{\Sigma_t}$ by using the estimate 
$$
\int_U\biggl|\int_U J(x,u,x',u')\,\de\sigma'(u')-\int_{U\times U} J(x,w,x',u')\,\de\sigma'(u')\de\sigma(w)\biggr|\,\de\sigma(u)\leq L_J\mathrm{diam}(U)
$$
which, thanks to the uniformity {\crd  with respect to} $y'=(x',\sigma')$, gives that $\| (b_{\Sigma_t})_\sigma\|_{TV}\leq L_J\mathrm{diam}(U)$.
Then, thanks to the discussion in Section~\ref{sec:diffcurves}, we obtain also $\mathrm C^1$ regularity with respect to the total variation norm.

\subsection[Structural properties of the interaction term]{Structural properties of the interaction term $f$}\label{sec:structprop}

Recall that $C=\mathbb{R}^d\times\Probabilities{U}$ is a closed and convex subset of $Y=\mathbb{R}^d\times F(U)$.
We shall denote in the sequel by $L_e$ and $L_J$ the Lipschitz constants of $e$ and $J$ respectively.

\begin{remark}\label{L5}
We endow $C\times C$ with the distance induced by the norm in
$Y\times Y$
\begin{displaymath}
  \|(y_1,y_2)\|\coloneqq\|y_1\|_{\rm BL}+\|y_2\|_{\rm BL}.
\end{displaymath}
With this choice, a function is $L$-Lipschitz if (and only if)
it is $L$-Lipschitz separately in the components, that is, if for every $y,\,y_1,\,y_2\in C$ one has
\begin{equation}\label{L10}
\norm{f(y_1,y)-f(y_2,y)}\leq L\norm{y_1-y_2}_{\rm BL},\qquad \norm{f(y,y_1)-f(y,y_2)}\leq L\norm{y_1-y_2}_{\rm BL}
\end{equation}
then
\begin{equation}\label{eq:37}
  \|f(y_1,y_2)-f(y_1',y_2')\|\le L\big(\|y_1-y_1'\|_{\rm BL}+\|y_2-y_2'\|_{\rm BL}\big)=L\|(y_1,y_2)-(y_1',y_2')\|.
\end{equation}
\end{remark}

For $y=(x,\sigma)\in C$ and $\Sigma\in\Probabilitiesone{C}$, let $b_\Sigma(y)=\int_U f(y,y')\,\de\Sigma(y')$ as in \eqref{pv2}, where $f\colon C\times C\to Y$ is defined by
\begin{equation}\label{pv511}
f(y,y')=\bigg(a(y),\Big( \int_U J(x,\cdot, x',u')\,\de\sigma'(u')-\int_{U\times U}\ J(x,w,x',u')\,\de\sigma'(u')\de\sigma(w)\Big)\sigma\bigg),
\end{equation}
with $a(y)=a(x,\sigma)=\int_U e(x,u)\,\de\sigma(u)$ as in \eqref{mm010}.
We define a map $j\colon\mathbb R^d\times\mathbb R^d\times\Probabilities{U}\to\Lip(U)$ by
\begin{equation}\label{pv518}
(x,x',\sigma')\mapsto j(x,x',\sigma')(u)\coloneqq\int_U J(x,u,x',u')\,\de\sigma'(u').
\end{equation}
Notice that the map $j(x,x',\sigma')$ depends linearly on $\sigma'$.
Moreover, recalling that $\sigma'$ is a probability measure, it is not difficult to see that 
\begin{equation}\label{pv519}
\norm{j(x,x',\sigma')}_{\Lip(U)}=\norm{j(x,x',\sigma')}_{L^\infty(U)}+\Lip(j(x,x',\sigma'))\leq
\max|J(x,\cdot,x',\cdot)|+L_J\diam(U).
\end{equation}

\begin{lemma}\label{pv516}
Let $\sigma\in\Probabilities{U}$ and let $z\in\Lip(U)$.
Then the following estimate holds
\begin{equation}\label{pv517}
\norm{z\sigma}_{\rm BL}\leq \norm{z}_{\Lip} \norm{\sigma}_{\rm BL}.
\end{equation}
\end{lemma}
\begin{proof}
It follows directly by \eqref{mm200bis} and by the definition \eqref{mm200} of the BL norm:
\begin{equation*}
\norm{z\sigma}_{\rm BL}=\sup_{\eta:\norm\eta_{\Lip}\leq1} \int_U \eta(u)z(u)\,\de\sigma(u)\leq\norm{z}_{\Lip}\norm{\sigma}_{\rm BL}. \qedhere
\end{equation*}
\end{proof}

\begin{proposition}\label{pv512}
The map $f$ defined in \eqref{pv511} is $L$-Lipschitz (in both variables), with $L$ depending only on
$L_e$, $L_J${\crd , and} $\mathrm{diam}(U)$, and satisfies the compatibility condition
{\crd  (see also \eqref{eq:34} below)
\begin{equation}
  \label{eq:34x}
  \forall\,R>0\ \exists\,\theta>0:\quad
  y,\,y'\in C\cap B_R(0)\quad\Rightarrow\quad
  y\pm\theta f(y,y')\in C.
\end{equation} }
\end{proposition}
\begin{proof}
By Remark \ref{L5}, we can study the Lipschitz dependence of $f$ separately with respect to $y$ and $y'$.
Moreover, we can consider the Lipschitz dependence on $x$ and $\sigma$ separately, keeping the other variable frozen.
Let us start with $f_x$; since it does not depend on $y'$, we only study the Lipschitz dependence on $y$.
We have
\begin{equation}\label{pv515}
\begin{split}
|f_x  (y_1, & y') -{}  f_x(y_2,y')| = |a(y_1)-a(y_2)| = \bigg|\int_U e(x_1,u)\,\de\sigma_1(u)-\int_U e(x_2,u)\,\de\sigma_2(u)\bigg| \\
\leq{} & \bigg|\int_U e(x_1,u)\,\de\sigma_1(u)-\int_U e(x_1,u)\,\de\sigma_2(u)\bigg| +\bigg|\int_U e(x_1,u)\,\de\sigma_2(u)-\int_U e(x_2,u)\,\de\sigma_2(u)\bigg| \\
\leq & L_e(1+\diam U)\norm{\sigma_1-\sigma_2}_{\rm BL}+L_e|x_1-x_2|\leq L_e(1+\diam U)\norm{y_1-y_2},
\end{split}
\end{equation}
where we have used \eqref{pv520}.
To study the Lipschitz dependence of $f_\sigma$ on its variables, it is convenient to do it for $x$, $x'$, $\sigma$, and $\sigma'$ separately.
The Lipschitz dependence on $x$ is easy to obtain, and it leads to
\begin{equation}\label{pv528}
\norm{f_\sigma(x_1,\sigma,y')-f_\sigma(x_2,\sigma,y')}_{\rm BL}\leq 2L_J(1+\diam U)|x_1-x_2|.
\end{equation}
Similarly, one can prove that
\begin{equation}\label{pv529}
\norm{f_\sigma(y,x_1',\sigma')-f_\sigma(y,x_2',\sigma')}_{\rm BL}\leq 2L_J(1+\diam U)|x_1'-x_2'|.
\end{equation}
Let us consider the dependence on $\sigma$.
Using the map $j(x,x',\sigma')$ defined in \eqref{pv518}, we have to estimate
\begin{equation}\label{pv521}
\begin{split}
\norm{f_\sigma(x,\sigma_1,y')-{} & f_\sigma(x,\sigma_2,y')}_{BL} =
\bigg\lVert \sigma_1\Big(j(x,x',\sigma')(\cdot)-\int_U j(x,x',\sigma')(v)\,\de\sigma_1(v)\Big) \\
& -\sigma_2\Big(j(x,x',\sigma')(\cdot)-\int_U j(x,x',\sigma')(v)\,\de\sigma_2(v)\Big)\bigg\rVert_{\rm BL} \\
\leq{} & \bigg\lVert (\sigma_1-\sigma_2) \Big(j(x,x',\sigma')(\cdot)-\int_U j(x,x',\sigma')(v)\,\de\sigma_2(v)\Big)\bigg\rVert_{\rm BL} \\
& + \bigg\lVert \sigma_1 \Big( \int_U j(x,x',\sigma')(v)\,\de(\sigma_2-\sigma_1)(v)\Big)\bigg\rVert_{\rm BL} \eqqcolon I+II.
\end{split}
\end{equation}
Term $I$ above can be estimated as follows
\begin{equation}\label{pv522}
\begin{split}
I={} & \bigg\lVert (\sigma_1-\sigma_2) \Big(\int_U
\big(j(x,x',\sigma')(\cdot)-j(x,x',\sigma')(v)\big)\,\de\sigma_2(v)\Big)\bigg\rVert_{\rm BL} \\
\leq{} & L_J(1+\diam U)\norm{\sigma_1-\sigma_2}_{\rm BL}.
\end{split}
\end{equation}
To estimate $II$, we use the definition of BL norm and the fact that $\sigma_1-\sigma_2\in\Measureso{U}$
\begin{equation}\label{pv523}
\begin{split}
II ={} & \sup_{\eta:\norm{\eta}_{\Lip}\leq1} \int_U \eta(u)\Big( \int_U j(x,x',\sigma')(v)\,\de(\sigma_2-\sigma_1)(v)\Big)\,\de\sigma_1(u) \\
\leq{} & \bigg|\int_U j(x,x',\sigma')(v)\,\de(\sigma_2-\sigma_1)(v)\bigg| \\
\leq{} & \int_U |j(x,x',\sigma')(v)-j(x,x',\sigma')(u)|\,\de(\sigma_2-\sigma_1)(v) \\
\leq{} & L_J\diam U\norm{\sigma_1-\sigma_2}_{\rm BL}.
\end{split}
\end{equation}
Putting \eqref{pv522} and \eqref{pv523} together, we can complete the estimate for \eqref{pv521} and obtain
\begin{equation}\label{pv524}
\norm{f_\sigma(x,\sigma_1,y')-f_\sigma(x,\sigma_2,y')}_{\rm BL}\leq 2L_J(1+\diam U)\norm{\sigma_1-\sigma_2}_{\rm BL}.
\end{equation}
Using \eqref{pv517} and that $\norm{\sigma}_{\rm BL}\leq1$, let us now estimate 
\begin{equation}\label{pv525}
\begin{split}
\norm{f_\sigma(y,x',\sigma_1')-{} & f_\sigma(y,x',\sigma_2')}_{\rm BL}
= \bigg\lVert \sigma\Big(j(x,x',\sigma_1')(\cdot)-
\int_U j(x,x',\sigma_1')(v)\,\de\sigma(v)\Big) \\
& -\sigma\Big(j(x,x',\sigma_2')(\cdot)-\int_U j(x,x',\sigma_2')(v)\,\de\sigma(v)\Big) \bigg\rVert_{\rm BL} \\
\leq{} & \norm{j(x,x',\sigma_1')-j(x,x',\sigma_2')}_{\Lip} \\
& + \bigg\lVert\int_U \big(j(x,x',\sigma_1')(v)-j(x,x',\sigma_2')(v)\big)\,\de\sigma(v)\bigg\rVert_{\Lip} \eqqcolon I'+II'.
\end{split}
\end{equation}
To estimate $I'$, we use the definition \eqref{pv518} of $j(x,x',\sigma')$ and the fact that $\sigma_1'-\sigma_2'\in\Measureso{U}$ to obtain
\begin{equation}\label{pv526}
\begin{split}
I' ={} & \bigg\lVert \int_U J(x,\cdot,x',u')\,\de\sigma_1'(u')-
\int_U J(x,\cdot,x',u')\,\de\sigma_2'(u') \bigg\rVert_{\Lip} \\
={} & \bigg\lVert \int_U J(x,\cdot,x',u')\,\de(\sigma_1'-\sigma_2')(u') \bigg\rVert_{\Lip} \\
={} & \bigg\lVert \int_U \big(J(x,\cdot,x',u')-J(x,\cdot,x',v')\big)\,\de(\sigma_1'-\sigma_2')(u') \bigg\rVert_{\Lip} \\
\leq{} & L_J(1+\diam U)\norm{\sigma_1'-\sigma_2'}_{\rm BL}.
\end{split}
\end{equation}
The estimate of $II'$ follows in a similar way, so that we obtain
\begin{equation}\label{pv527}
\norm{f_\sigma(y,x',\sigma_1')- f_\sigma(y,x',\sigma_2')}_{\rm BL}\leq 2L_J(1+\diam U)\norm{\sigma_1'-\sigma_2'}_{\rm BL}.
\end{equation}
Putting \eqref{pv528}, \eqref{pv529}, \eqref{pv524}, and \eqref{pv527} together, we obtain
\begin{equation}\label{530}
\norm{f_\sigma(y_1,y_1')-f_\sigma(y_2,y_2')}_{\rm BL}\leq 2L_J(1+\diam U)\big(\norm{y_1-y_2}+\norm{y_1'-y_2'}\big),
\end{equation}
which, together with \eqref{pv515} gives the Lipschitz estimate on $f$.\\
Let us now discuss the compatibility conditions \eqref{eq:34} for the $f$ defined in \eqref{pv511}. 
It is clear that the first component of $y+\theta f(y,y')$, namely $x+\theta a(y)$, belongs to $\mathbb{R}^d$ for all $\theta\in\mathbb{R}$, so that we are left with checking that the second component $\sigma+\theta f_\sigma(y,y')$, namely 
\begin{equation}\label{pv514}
\sigma+\theta\sigma\Big( \int_U J(x,\cdot,x',u')\,\de\sigma'(u')-\int_{U\times U} J(x,u,x',u')\,\de\sigma'(u')\de\sigma(u)\Big),
\end{equation}
belongs to $F(U)$.
As a matter of fact, we will prove that \eqref{pv514} is an element of $\Probabilities{U}$, which means that its integral over $U$ is $1$ and that it is positive.
The proof that $\sigma+\theta f_\sigma(y,y')\geq 0$ can be obtained via some manipulations and using the Lipschitz estimate on $J$.
Indeed,
\begin{equation*}
\begin{split}
\sigma+\theta f_\sigma(y,y') & = \sigma\bigg(1+\theta\Big( \int_U
J(x,u, x',u')\,\de\sigma'(u')-\int_{U\times U} J(x,w,x',u')\,\de\sigma'(u')\de\sigma(w)\Big)\bigg) \\
&= \sigma\bigg(1+\theta \int_U \Big(J(x,u,x',u')-\int_U J(x,w,x',u')\,\de\sigma(w)\Big)\,\de\sigma'(u')\bigg) \\
&= \sigma\bigg(1+\theta \int_{U\times U} (J(x,u,x',u')-J(x,w,x',u'))\,\de\sigma(w)\de\sigma'(u')\bigg) \\
&\geq \sigma\bigg(1-\theta L_J\int_U d_U(u,w)\,\de\sigma'(u')\bigg) \geq \sigma(1-\theta L_J\diam U),\\
\end{split}
\end{equation*}
which is nonnegative as soon as $\theta\leq (L_J\diam U)^{-1}$.
By recalling that $f_\sigma(y,y')\in\Measureso{U}$, we obtain that $\sigma+\theta f_\sigma(y,y')\in\Probabilities{U}$.
\end{proof}

\section{Existence and uniqueness of Lagrangian solutions}\label{sec:existuniq}

\subsection{Interaction systems in Banach spaces}\label{subsecint}

Let us consider now a Banach space $(Y,\|\cdot\|)$ with a closed
convex set $C$ and a $L$-Lipschitz map
\begin{equation}
  \label{eq:28}
  f\colon C\times C\to Y
\end{equation}
satisfying the compatibility condition
\begin{equation}
  \label{eq:34}
  \forall\,R>0\ \exists\,\theta>0:\quad
  c,\,c'\in C\cap B_R(0)\quad\Rightarrow\quad
  c+\theta f(c,c')\in C.
\end{equation}
Let us consider a continuous curve of measures
$\Sigma\in \mathrm{C}([0,T];(\Probabilitiesone{C},W_1))$.
Recalling that 
\begin{equation}\label{pv500}
  \|f(c,c')\|\le \norm{f(c_0,c_0)}+L(\|c-c_0\|+\|c'-c_0\|)
\end{equation}
where $c_0$ is an arbitrary point in $C$, 
we can define the time-dependent vector field $b_\Sigma(t,\cdot)\colon C\to Y$ by
\begin{equation}\label{pv2bis}
  b_\Sigma(t,c)=b_{\Sigma_t}(c)\coloneqq\int_C f(c,c')\,\de\Sigma_t(c'),
\end{equation}
where the integral above can be interpreted in the strong sense, as a Bochner integral.

We are going to prove the following result, which provides (taking Proposition~\ref{pv512} into account)
 the proof of Theorem~\ref{thm:main}.

\begin{theorem}
  \label{prob:mainBanach}
  Given $\bar\Sigma\in \Probabilitiesone C$ 
  there exists $\Sigma\in \mathrm{C}^0([0,T];(\Probabilitiesone{C},W_1))$ with $\Sigma_0=\bar\Sigma$ such that the
   family of transition maps {\crd  $\bY_\Sigma(t,s,\cdot)$} in $C$
  associated to the ODE
  \begin{equation}
    \label{eq:36}
    \dot y_t=b_\Sigma(t,y_t),\quad y_s=y,\quad y\in 
    \rmC^1([s,T];Y),\ y_t\in
    C\quad\text{for every }t\in [s,T]
  \end{equation}
 with the vector field $b_\Sigma$ given by \eqref{pv2bis} satisfies
\begin{equation}\label{pv4bis}
  \Sigma_t=\PushForward{\bY_\Sigma(t,s,\cdot)}{\Sigma_s}
  \qquad\text{for every $0\le s\leq t\le T$.}
\end{equation}
In addition, one has the stability estimate
\begin{equation}
\label{eq:stability}
W_1(\Sigma_t,\Sigma_t')\leq\mathrm e^{2L(t-s)}W_1(\Sigma_s,\Sigma_s'),\qquad\text{for every $0\leq s\leq t\leq T$},
\end{equation}
for the solutions $\Sigma,\,\Sigma'$ starting from $\bar\Sigma$, $\bar\Sigma'$, where $L$ is the 
Lipschitz constant of $f$.
\end{theorem}

\subsection{Existence for the discrete problem}

We first study the discrete problem for $N$ particles evolving in $Y$, corresponding to the evolution of a discrete (atomic) measure.  This case could be simply seen as a byproduct of the more general ``diffuse'' measure well-posedness result; however we include it both as a guideline to introduce the more general case and also as a constructive approximation (for $N$ large, see also Remark~\ref{N->oo} below), which could be useful for the purpose of numerical simulation.
We consider the convex set $C^N$ in $Y^N$ with the norm
\begin{displaymath}
  \|\by\|_{Y^N}\coloneqq\frac 1N\sum_{i=1}^N \|y_i\|,\qquad
  \by=(y_1,\ldots,y_N)\in Y^N.
\end{displaymath}
We define the map $\boldsymbol f^N=(f_1^N,\ldots,f_N^N)\colon C^N\to Y^N$ by
\begin{equation}\label{pv8}
f_i^N(\by)=\frac1N\sum_{j=1}^N f(y_i,y_j),\qquad i=1,\ldots,N.
\end{equation}
We notice that $\boldsymbol f$ is Lipschitz, since
\begin{align*}
  \|\boldsymbol f^N(\by)-\boldsymbol f^N(\by')\|_{Y^N}
  &\le 
  \frac1N\sum_{i=1}^N\|f_i^N(\by)-f_i^N(\by')\|
  \\
  &\le 
    \frac1{N^2}\sum_{i,\,j=1}^N \|f(y_i,y_j)
    -f(y_i',y_j')\|
    \le  
    \frac L{N^2}\sum_{i,\,j=1}^N
    (\|y_i-y_i'\|+\|y_j-y_j'\|)
  \\&\le 2L\|\by-\by'\|_{Y^N}.
\end{align*}
Let us now check that $C^N$ satisfies the invariance properties with respect to $\boldsymbol f^N$: if $\by\in C^N$ with $\|\by\|\le R$ then 
every component $y_i$ belongs to $C$ and $\|y_i\|\le NR$.
By \eqref{eq:34} (applied to the constant $NR$) we may find a constant $\theta>0$ such that 
\begin{displaymath}
  y_i+\theta f(y_i,y_j)\in C\qquad\text{for every $i,j$,}
\end{displaymath}
so that the convexity of $C$ yields
\begin{displaymath}
  y_i+\theta f^N_i(\by)=
  \frac1N\sum_{j=1}^N \big(y_i+\theta f(y_i,y_j)\big)\in C.
\end{displaymath}

By applying Theorem~\ref{thm:Brezis} with $C^N$, $Y^N$ we obtain the following result.
\begin{corollary}
  For every $\bar\by\in C^N$, there exists a unique curve 
  $\by\colon[0,+\infty)\to C^N$ of class $\rmC^1$ such that
  \begin{equation}\label{pv1}
    \begin{cases}
      \dot\by(t)=\bff^N(\by^N(t)), &\\
      \by(0)=\bar\by.
    \end{cases}
  \end{equation}
  In particular, the family $\Sigma_t\coloneqq\frac 1N\sum_{i=1}^N
  \delta_{y_{i,t}}$ provides a solution to the existence part of Theorem~\ref{prob:mainBanach}
  for the initial datum $\bar\Sigma\coloneqq\frac 1N\sum_{i=1}^N\delta_{\bar y_i}$.
\end{corollary}

\subsection{Stability estimates}\label{L3}
\begin{proposition}[Properties of $b_t$]\label{pv5}
Let $\Lambda,\,\Lambda'\in \rmC([0,T];\Probabilitiesone C)$, let
$b_\Lambda\,,b_{\Lambda'}$ be defined as in \eqref{pv2bis} and let $f\colon C\times C\to Y$ be
$L$-Lipschitz.
Then
\begin{itemize}
\item[(i)] $\norm{b_\Lambda(t,y)}\leq \norm{f(y_0,y_0)}+L\norm{y-y_0}+L\int_C \norm{y'-y_0}\,\de\Lambda_t(y')$ for all $y_0\in Y$;
\item[(ii)] $\norm{b_\Lambda(t,y)-b_\Lambda(t,z)}\leq L\norm{y-z}$;
\item[(iii)] $\norm{b_\Lambda(t,y)-b_\Lambda(s,y)}\leq L W_1(\Lambda_t,\Lambda_s)$;
\item[(iv)] $\norm{b_{\Lambda}(t,y)-b_{\Lambda'}(t,y)}\leq L
  W_1(\Lambda_t,\Lambda_t')$;
\item[(v)]
  If there exists $\bar R>0$ such that 
  $\Lambda_t(C\setminus B_{\bar R}(0))=0$ for every $t\in [0,T]$,
  then for every $R>0$ there exists $\theta>0$ such that 
  \begin{equation}
    \label{eq:38}
    y\in C,\ \|y\|\le R,\ t\in [0,T]\quad
    \Rightarrow\quad
    y+\theta b_\Lambda(t,y)\in C.
  \end{equation}
\end{itemize}
\end{proposition}
\begin{proof}
Property (i) follows immediately from \eqref{pv500}.
To prove (ii), we notice that 
\begin{align*}
  \|b_\Lambda(t,y)-b_\Lambda(t,z)\|
  &=
    \left\|
    \int_C \Big(f(y,y')-f(z,y')\Big)\,\de\Lambda_t(y')\right\|
    \le L\|y-z\|.  
\end{align*}
Estimate (iii) is a simple computation
\begin{equation*}
\begin{split}
\norm{b_\Lambda(t,y)-b_\Lambda(s,y)}&=  \left\lVert\int_C
  f(y,y')\,\de(\Lambda_t-\Lambda_s)(y')\right\rVert
\\&=
\sup_{z\in {\crd Y'},\ \|z\|_{ {\crd Y'}}\le 1}\int_C \langle z,f(y,y')\rangle\,\de(\Lambda_t-\Lambda_s)(y') \\
&\leq L W_1(\Lambda_t,\Lambda_s),
\end{split}
\end{equation*}
where we have used that the map $y'\mapsto\langle z,f(y,y')\rangle$ is $L$-Lipschitz.
The proof of (iv) is analogous.

Let us now consider the last statement (v); we may assume
$R\ge \bar R$ and we can choose $\theta>0$ such that \eqref{eq:34}
holds. Therefore
\begin{displaymath}
  y+\theta b_\Lambda(t,y)=
  \int_C \Big(y+\theta f(y,y')\Big)\,\de\Lambda_t(y')\in C,
\end{displaymath}
since $C$ is convex and closed, and $\Lambda_t$ is a probability measure.
\end{proof}

\begin{corollary}
  Let $\Lambda,\,\Lambda^i\in \rmC([0,T];\Probabilitiesone C)$, let
  $b_\Lambda,\,b_{\Lambda^i}$ be defined as in \eqref{pv2bis}, $y_0\in Y$
  and let $f\colon C\times C\to Y$ be $L$-Lipschitz. Then
\begin{enumerate} 
\item for every $y\in C$ and $s\in [0,T]$  there exists a unique solution
  $y_t=\bY_\Lambda(t,s,y)$ from $[s,T]$   to $C$ of class $\rmC^1$ of the Cauchy problem
  \begin{equation}
    \label{eq:39}
    \dot y_r=b_\Lambda(r,y_r)=b_{\Lambda_r}(y_r),\qquad
    y_s=y;
  \end{equation}
\item $\bY_\Lambda(t,0,y)$ satisfies the estimate
\begin{equation}\label{pv502}
\norm{\bY_\Lambda(t,0,y)-y_0}\leq \big(\norm{y-y_0}+tB(\Lambda,t,{\crd  y_0})\big)\mathrm{e}^{Lt}
\end{equation}
with 
\begin{equation}\label{pv501}
B(\Lambda,t,y_0)\coloneqq\norm{f(y_0,y_0)}+L\max_{s\in[0,t]}\int_C \norm{y'-y_0}\,\de\Lambda_s(y');
\end{equation}
\item $\bY_\Lambda(\cdot,0,y)$ satisfies the estimate
\begin{equation}\label{pv503}
\!\!\! \norm{\bY_\Lambda(t,0,y)-\bY_\Lambda(t',0,y)}\leq |t-t'|\big[B(\Lambda,T,y_0)+L\big(\norm{y-y_0}+TB(\Lambda,T,y_0)
\big)\mathrm{e}^{LT}\big];
\end{equation}
\item  $\bY_\Lambda(t,s,\cdot)$ satisfies the estimate
  \begin{equation}
    \label{eq:40}
    \|\bY_\Lambda(t,s,y)-\bY_\Lambda(t,s,y')\|
    \le \mathrm e^{L(t-s)}\,\|y-y'\|, \qquad 0\le s\le t\le T;
  \end{equation}
\item more generally, $\bY_{\Lambda^1},\,\bY_{\Lambda^2}$ satisfy the
  estimate for $0\le s\le t\le T$:
\begin{equation}
    \label{eq:40}
    \|\bY_{\Lambda^1}(t,s,y^1)-\bY_{\Lambda^2}(t,s,y^2)\|
    \le \mathrm e^{L(t-s)}\,\|y^1-y^2\|+
    L\int_s^t \mathrm e^{L(t-\tau)} W_1(\Lambda^1_\tau,\Lambda^2_\tau)\,\de\tau.
  \end{equation}
\end{enumerate}
\end{corollary}
\begin{proof}
  Let us first assume that $\Lambda$ and $\Lambda^i$ are
concentrated on a ball of radius $\bar R$ in $Y$.
Then statements (i) and (iv) immediately follow by Theorem~\ref{thm:Brezis}, thanks to the estimates of Proposition \ref{pv5}.

The proof of (ii) is a computation:
recalling \eqref{pv2}, Proposition~\ref{pv5}(i) and the definition \eqref{pv501} of $B(\Lambda,t,y_0)$,
the triangle inequality gives
\begin{equation}
\begin{split}
\norm{\bY_\Lambda(t,0,y)-y_0}\leq & \norm{y-y_0}+\int_0^t \norm{b_{\Lambda_s}(y(s))}\,\de s \\
\leq &  \norm{y-y_0}+tB(\Lambda,t,y_0)+L\int_0^t \norm{y(s)-y_0}\,\de s,
\end{split}
\end{equation}
which yields \eqref{pv502} by Gronwall's inequality.
Estimate \eqref{pv503} follows from combining Proposition~\ref{pv5}(i) with estimate \eqref{pv502}, so that (iii) is proved.

Concerning (v), it is clearly sufficient to consider the case $s=0$.
Denoting by $y^i_t$, $t\in [0,T]$, the solutions to \eqref{eq:39} with respect to the fields $b_t^i=b_{\Lambda_t^i}$ 
and the  initial conditions $\bar y^i$, from $\dot y^i=b_t^i(y^i)$ we have, by Proposition~\ref{pv5}(ii) and (iv)
$$
\frac\de{\de t}\norm{y^1-y^2}(t)\leq 
\norm{b^1_t(y^1)-b^1_t(y^2)}+\norm{b_t^1(y^2)-b^2_t(y^2)}\leq
L\norm{y^1-y^2}+L W_1(\Lambda_t^1,\Lambda_t^2),$$
which gives by a simple comparison argument
\begin{equation}\label{pv18}
\norm{y^1(t)-y^2(t)}\leq 
\mathrm e^{Lt}\norm{y^1-y^2}+L\int_0^t \mathrm e^{L(t-\tau)}W_1(\Lambda_\tau^1,\Lambda_\tau^2)\,\de\tau.
\end{equation}
The general case when $\Lambda$ may have unbounded support can be obtained by approximation, using once
more Proposition~\ref{pv5}, since the estimates
are independent of $\bar R$.
\end{proof}

\subsection{Contractivity and stability}
We now fix  $\bar \Sigma\in \Probabilitiesone C$ 
and we consider the metric space
\begin{equation}
  \label{eq:41}
  \mathscr A\coloneqq\big\{\Lambda\in \rmC([0,T];(\Probabilitiesone C,W_1)):
  \Lambda_0=\bar \Sigma\big\},
\end{equation}
complete when endowed with the usual sup distance (as a consequence
of the completeness of $(\Probabilitiesone C,W_1)$). We define a map $\mathcal T\colon\mathscr A\to\mathscr A$ in the
following way:
given $\Lambda\in \mathscr A$ we first compute the flow map
$\bY_\Lambda(t,s,\cdot)$ associated to $b_\Lambda$ and then we define
the curve $\mathcal T[\Lambda]\colon [0,T]\to\Probabilitiesone C$ by
\begin{equation}\label{eq:42}
\mathcal T[\Lambda]_t\coloneqq \bY_\Lambda(t,0,\cdot)_\#\bar\Sigma
\end{equation}
It is immediate to check that $\mathcal T$ maps $\mathscr A$ to $\mathscr A$.

\begin{lemma}
For every $\Lambda, \,\Lambda^1,\,\Lambda^2\in \mathscr A$ we have
\begin{equation}\label{pv504}
W_1(\mathcal T[\Lambda]_t,\mathcal T[\Lambda]_s)\leq |t-s| \bigg[B(\Lambda,T,y_0)+
L\bigg(\int_C\norm{y-y_0}\,\de\bar\Sigma(y)+TB(\Lambda,T,y_0)\bigg)\mathrm{e}^{LT}\bigg],
\end{equation}
\begin{equation}
    \label{eq:43} 
    W_1(\mathcal T[\Lambda^1]_t,\mathcal T[\Lambda^2]_t)\le 
    L\int_0^t \mathrm e^{L(t-\tau)}W_1(\Lambda^1_\tau,\Lambda^2_\tau)\,\de \tau,
  \end{equation}
  where the constant $B(\Lambda,T,y_0)$ is defined in \eqref{pv501} for $t=T$.
\end{lemma}
\begin{proof}
Estimate \eqref{pv504} follows immediately from the definition of $\mathcal T[\Lambda]$ in \eqref{eq:42}, 
estimate \eqref{pv503}, and the fact that $\bar\Sigma$ is a probability measure.
Estimate \eqref{eq:43} is a direct consequence of \eqref{eq:40}, since
\begin{equation*}
W_1(\mathcal T[\Lambda^1]_t,\mathcal T[\Lambda^2]_t)\le \int_C \|\bY_{\Lambda^1}(t,0,y)-\bY_{\Lambda^2}(t,0,y)\|\,\de\bar\Sigma(y)
\end{equation*}
and $\bar\Sigma$ is a probability measure.
\end{proof}

\begin{corollary}
  The map $\mathcal T$ admits a unique fixed point, which provides the 
  unique solution $\Sigma$ in Theorem~\ref{prob:mainBanach}.
\end{corollary}
\begin{proof}
  Let us fix a constant $L'>2L$ so that $\ell\coloneqq L/(L'-L)<1$ and let us consider the equivalent distance in $\mathscr A$ given by
  $$
  \sfd(\Lambda,\Lambda')\coloneqq \max_{t\in [0,T]}\mathrm e^{-L't}W_1(\Lambda_t,\Lambda_t').
  $$
  Then, from \eqref{eq:43} we immediately get
  $$
  \mathrm e^{-L't}W_1(\mathcal T[\Lambda]_t,\mathcal T[\Lambda']_t)\leq 
  L\int_0^t \mathrm e^{(L-L')(t-s)}\,\de s \,
  W_1(\Lambda_r,\Lambda_r')
  \quad\forall t\in [0,T],
  $$
  so that our choice of $L'$ gives $\sfd(\mathcal T[\Lambda],\mathcal T[\Lambda'])\leq \ell\sfd(\Lambda,\Lambda')$.
\end{proof}

We can slightly modify the previous argument in order to derive a 
stability estimate of the solution $\Sigma_t$ in terms of the initial
datum $\bar \Sigma$. 
\begin{lemma}
  Let $\bar\Sigma^1,\,\bar\Sigma^2$ be initial data in $\Probabilitiesone C$ and
  let $\Sigma^i_t$ be the corresponding solutions. Then
  \begin{equation}
    \label{eq:44}
    W_1(\Sigma^1_t,\Sigma^2_t)\le \mathrm e^{2L t}\,W_1(\bar\Sigma^1,\bar\Sigma^2)\qquad\forall t\in [0,T].
  \end{equation}
\end{lemma}
\begin{proof}
  Let us fix $0\le s<t\le T$ and 
  consider an optimal coupling $\Pi_s$ between $\Sigma_s^1$ and $\Sigma_s^2$, so that 
  using the fact that $\PushForward{\big(\bY_{\Sigma^1}(t,s,y^1),\bY_{\Sigma^2}(t,s,y^2)\big)}{\Pi_s}$ 
  is a coupling between $\Sigma_t^1$ and $\Sigma_t^2$
  we can write
\begin{equation}\label{pv17}
  W_1(\Sigma_t^1,\Sigma_t^2)\leq \int_{C\times C} \norm{\bY_{\Sigma^1}(t,s,y^1)-\bY_{\Sigma^2}(t,s,y^2)}\,\de\Pi_s(y^1,y^2).
\end{equation}
By \eqref{eq:40} we get 
\begin{equation*}\label{pv20}
\begin{split}
W_1(\Sigma_t^1,\Sigma_t^2)\leq{} & 
\mathrm e^{L(t-s)}\int_{C\times C} \norm{y^1-y^2}\,\de\Pi_s(y^1,y^2)+ L\int_s^t\mathrm e^{L(t-\tau)}W_1(\Sigma_\tau^1,\Sigma_\tau^2)\,\de\tau\\
={}& \mathrm e^{L(t-s)}W_1(\Sigma_s^1,\Sigma_s^2)+ L\int_s^t\mathrm e^{L(t-\tau)}W_1(\Sigma_\tau^1,\Sigma_\tau^2)\,\de\tau.
\end{split}
\end{equation*}
Choosing $t=s+h$, this proves that the upper right derivative $\de/\de s_+$ of the map $s\mapsto
W_1(\Sigma_s^1,\Sigma_s^2)$ satisfies
$$\frac{\de}{\de s_+} W_1(\Sigma_s^1,\Sigma_s^2)\leq 2L
W_1(\Sigma_s^1,\Sigma_s^2)$$
 and therefore $W_1(\Sigma_s^1,\Sigma_s^2)\leq \mathrm
e^{2Ls}W_1(\bar\Sigma^1,\bar\Sigma^2)$,
which proves \eqref{eq:44}. \qedhere
\end{proof}

\begin{remark}[Another existence proof] \label{N->oo}
  {\rm The following argument can provide an alternative strategy to the construction a
solution starting from the discrete solutions of the previous section.
In fact, one can use the contractivity to pass to the limit in the discrete problem.
Choose $\bar y_i(\omega)\in C$ independent and identically distributed, with law
$\bar\Sigma$, so that the random measures
$\bar\Sigma^N(\omega)\coloneqq\frac1N\sum_{i=1}^N\delta_{\bar y_i(\omega)}$ almost surely converge 
in $\Probabilitiesone{C}$ to $\bar\Sigma$.
We fix $\omega$ such that this happens, set $\bar y_i(\omega)=\bar y_i$ and let $\by(t)=(y_1(t),\ldots,y_N(t))$ be
the discrete evolution starting from $\by(t)=(\bar y_1,\ldots,\bar y_N)$. 
Then, by contractivity,  $\Sigma^N_t\coloneqq \frac1N\sum_{i=1}^N\delta_{y_i(t)}$ converges weakly, and it
is not hard to prove that $\Sigma_t=\lim_{N\to\infty}\Sigma_t^N$ provides a solution.
}\end{remark}

\section{Uniqueness of Eulerian solutions}\label{sec:euuni}

In this section we address the uniqueness of Eulerian solutions, according to \eqref{eq:32}. Our first proof uses
a classical duality argument, adapted to the infinite-dimensional space of measures and to the special structure
\begin{equation}\label{eq:recall}
b_\Sigma(t,c)=b_{\Sigma_t}(c)=\int_C f(c,c')\,\de\Sigma_t(c')
\end{equation}
of the vector field, with $f=(f_x,f_\sigma)$ as in
\eqref{eq:21}, \eqref{eq:21bis}. One of the advantages of the duality proof is that it provides uniqueness
in the larger class of signed measures; the drawback is that, since we don't have at our disposal the mollification
schemes of the finite-dimensional setting, we have to require $\rmC^1$ regularity in place of Lipschitz regularity
with respect to the $x$ variable of $e$ and $J$. We use the special structure of the vector field,
together with \eqref{eq:34x}, also to make use of the flow map $\YY(s,t,x)$ backward in time, i.e. for times $s\leq t$; 
indeed, \eqref{eq:34x} yields
that the abstract compatibility condition \eqref{eq:34} holds also for $-b$, whose forward solutions correspond to backward solutions for $b$.

In this section we shall apply the abstract calculus tools of Sections~\ref{sec:diff} and \ref{sec:regularity} with $C=\R^d\times\Probabilities{U}$,
$E=Y=\R^d\times F(U)$, the subspace $E_C=Y_C=\R^d\times\Measures_0(U)$ and its closure
$\overline{E}_C=\R^d\times\{\sigma\in F(U): \sigma(1)=0\}$.}

\subsection{Uniqueness by duality}

\begin{theorem}\label{thm:uni} 
Suppose that $J\colon(\R^d\times U)^2\to \R$ and $e\colon\R^d\times U\to\R$ are Lipschitz maps,
with $J(\cdot,u,x',u')$ of class $\rmC^1$ for all $(u,x',u')\in U\times\R^d\times U$, and 
$e(\cdot,u)$ of class $\rmC^1$ for all $u\in U$.
Then, for all $\bar\Sigma\in\Measures (C)$ with $\int_C |x|\,\de |\bar\Sigma|<+\infty$,
equation \eqref{eq:32} admits a unique solution in the class of weakly continuous maps $t\in [0,T]\mapsto
\Sigma_t\in \Measures (C)$ with $\sup_t\int_C (1+|x|)\,\de |\Sigma_t|<+\infty$ and $\Sigma_0=\bar\Sigma$.
\end{theorem}
\begin{proof}
Let us consider solutions $\Sigma^1,\Sigma^2$ of \eqref{eq:32} such that $\Sigma^1_0=\Sigma^2_0$ and fix $h>0$. 
Let us denote 
$$
{\cal R}\coloneqq \left\{r\colon[0,h]\to\rmC_b(Y)\colon \text{$r$ is Borel, $r_t\in\rmC^1_b(Y)$, $\mathrm{Lip}(r_t)\leq 1$
for all $t\in [0,h]$}\right\},
$$ 
with the usual notation $r_t(y)=r(t,y)$.

The difference $\Sigma_t\coloneqq\Sigma_t^1-\Sigma^2_t$ solves
\begin{equation}\label{118}
\frac{\de}{\de t}\Sigma_t+{\rm div}(b_{\Sigma_t^1}\Sigma_t)=-{\rm div}(b_t\Sigma^2_t)
\end{equation}
in the weak sense of \eqref{eq:32}, with $b_t\coloneqq b_{\Sigma_t^1}-b_{\Sigma^2_t}$.
Let us stress that $b_t$ has null first component, since it is the difference of the vector fields 
$b_{\Sigma_t^1}$ and $b_{\Sigma^2_t}$ which have the {\it same} 
first component (recall the notation \eqref{astnotation}):
$$
b_{\Sigma_t^i} (x,\sigma)= \bigg(\int_U e(x,u) \,\de\sigma(u), 
\Big( J\ast\Sigma_t^i(x,\cdot) - \int_U J\ast\Sigma_t^i(x,u) \,\de\sigma(u)\Big)\sigma\bigg), \quad i=1,2.
$$
By linearity of the second components with respect to $\Sigma$ we obtain the representation
$$
b_t (x,\sigma)= \bigg(0, \Big( J\ast\Sigma_t(x,\cdot) - \int_U J\ast\Sigma_t(x,u) {\crd  \, \mathrm{d}}\sigma(u)\Big)\sigma\bigg).
$$

We consider a bounded solution $g\in\rmC^1([0,h]\times C)$ of the backward transport equation with velocity 
field $b_{\Sigma_t^1}$, right-hand side $r_t$, and terminal condition $g_h=0$ (with the usual
notation $g_s(c)=g(s,c)$):
\begin{equation}\label{111}
\begin{cases}
\partial_t g_t+\mathrm D g_t (b_{\Sigma_t^1})=r_t & \text{in $[0,h]\times C$}, \\
g_h=0 & \text{in $C$}.
\end{cases}
\end{equation}
We stress now that the differentiations acting on $g_t$ with respect to $c$ are all meant in the Fr\'echet sense 
(see Definition~\ref{def:FD}) and that the pairing 
$\mathrm D g_t (b_{\Sigma_t^1})$ corresponds to the directional derivative of $g_t$ along the vector $b_{\Sigma_t^1} \in Y_C$.
By the classical method of characteristics, one can construct a solution $g$ of \eqref{111} with the required $\rmC^1$ regularity
property by setting:
\begin{equation}\label{116}
g_t(c)=-\int_t^h r_s(\YY_{\Sigma^1}(s,t,c)) \,\de s.
\end{equation}
Indeed, the $\rmC^1_b(Y)$ regularity of $r_t$, the $\rmC^1$ regularity of $\YY_{\Sigma^1}(s,t,\cdot)$ 
granted by Theorem~\ref{thm:extraA} (and the arguments below, see before formula \eqref{eq:superextra}), and 
Theorem~\ref{chainrule} yield the $\rmC^1$ regularity of the function $g$ in \eqref{116}, together with the
exchange of Fr\'echet differentiation with integration (the latter granted by \eqref{exdiffint}).
Obviously $g_h=0$ and one can check that $g$ satisfies \eqref{111} with the following observation: since, 
{thanks to \eqref{eq:chain1}}, 
$$
\frac{\de}{\de t}g(t,c(t))=\bigl(\partial_t g_t+\mathrm D g_t (b_{\Sigma_t^1})\bigr)(t,c(t))
$$
along {\it any} $\rmC^1$ solution $c(t)$ of the ODE 
$$
\frac{\de}{\de t}c(t)=b_{\Sigma_t^1}(c(t)),
$$
if for any $c_0\in C$ and $t_0\in (0,h)$ we are able to find a $\rmC^1$
solution $c(t)$ to the ODE above with $c(t_0)=c_0$ and $\frac{\de}{\de t}g(t,c(t))=r_t(c(t))$ at $t=t_0$, we are done. Choosing
$c(t)=\YY_{\Sigma^1}(t,0,d)$ for some $d\in C$, from the semigroup property we get
$$
g(t,c(t))=-\int_t^hr_s(\YY_{\Sigma^1}(s,0,d))\,\de s
$$
so that we are able to check \eqref{111} at any $(t_0,c_0)$ with $c_0=\YY_{\Sigma^1}(t_0,0,d)$. Choosing 
$d=\YY_{\Sigma^1}(0,t_0,c_0)$ (only at this point we are using the flow backwards in time) we obtain the
global validity of \eqref{111}.

As mentioned above, the $\rmC^1$ regularity of $\YY_{\Sigma^1}(s,t,\cdot)$ follows from
Theorem~\ref{thm:extraA}, if we check that
\begin{equation}\label{eq:superextra}
\text{$(t,c)\mapsto\mathrm D_{Y_C}A(t,\cdot)(c)$ is continuous from $[0,T]\times C$ to $\mathcal L(Y_C,\overline{Y}_C)$}
\end{equation}
in the sense of \eqref{extraA}, with
$$
A(t,c)=b_{\Sigma^1_t}(c),
$$
and we recall Proposition~\ref{pv5}.
It is at this stage that we need the extra $\rmC^1$ assumption on $J(\cdot,u,x',u')$  and $e(\cdot,u)$. Indeed, thanks to the representation
\eqref{eq:recall} of $b_{\Sigma^1}$, it is
sufficent to check $\rmC^1$ differentiability of $f(\cdot,c')$ for all $c'=(x',\sigma')$ in the direction
$(v,\theta)\in Y_C$; the partial differential with respect to the $x$ variable
at $c=(x,\sigma)$ is given by
\begin{equation*}
\begin{split}
v\in\mathbb R^d \mapsto\bigg( & \int_U \mathrm D_x e (x,u)(v)\,\de\sigma(u), \\
& \Bigl(\int_U \mathrm D_x J(x,\cdot,x',u')({v)}\,\de\sigma'(u')-\int_{U\times U}\mathrm D_x J(x,w,x',u')
({v)}\,\de\sigma'(u')\de\sigma(w)\Bigr)\sigma\bigg)
\end{split}
\end{equation*}
while a partial differential with respect to the $\sigma$ variable is given by
\begin{equation*}
\begin{split}
\theta \in {\cbl{}\Measures_0(U)}\mapsto\bigg(\int_U e(x,u)\,\de\theta(u), &
\Bigl(\int_U J(x,\cdot,x',u')\,\de\sigma'(u')-\int_{U\times U}J(x,w,x',u')\,\de\sigma'(u')\de\sigma(w)\Bigr)\theta \\
& -\Bigl(\int_{U\times U}J(x,w,x',u')\,\de\sigma'(u')\de\theta(w)\Bigr)\sigma\bigg),
\end{split}
\end{equation*}
so that the differential is continuous from $[0,T]\times C$ to $\mathcal L(Y_C,\overline{Y}_C)$.

Since $r_t$ are $1$-Lipschitz, and the Lipschitz constant of $b_{\Sigma^1_t}$ can be estimated from above by
$L_*=L\sup_t|\Sigma^1_t|(C)$, with $L$ Lipschitz constant of the interaction term $f$ in \eqref{eq:recall}, from
the Lipschitz estimate on $\bY_{\Sigma^1}(t,s,\cdot)$ granted by Theorem~\ref{thm:extraA} we get 
\begin{equation}\label{inisup}
\|\mathrm D g_t(c)\|_{\mathcal L(Y_0,Y)}\leq (h-t)\mathrm e^{L_*T},\qquad\text{for all $(t,c)\in [0,h]\times C$.}
\end{equation}
 
Since $g_t\Sigma_t$ vanishes at $t=0$ and $t=h$, by the Leibniz rule and Theorem~\ref{chainrule} we get
\begin{equation}\label{112}
0=\int_0^h \frac{\de}{\de t}\int_C g_t\,\de\Sigma_t\,\de t=\int_0^h\int_C\bigl(\partial_tg_t+\mathrm D g_t (b_{\Sigma_t^1})\,\de\Sigma_t\,\de t + 
\int_0^h\int_C \mathrm D g_t (b_{\Sigma_t^1})\,\de \Sigma^2_t\,\de t .
\end{equation}
Using \eqref{111}, we obtain that
\begin{equation}\label{firstsup}
\sup_{r\in {\cal R}}\int_0^h\int_C r_t\,\de\Sigma_t\,\de t \leq \int_0^h \int_C |\mathrm Dg_t(b_t)| \,\de |\Sigma^2_t|\,\de t.
\end{equation}

Motivated by this estimate, we work with the distance $\sfd_{{\crd  \mathrm C^1}}$ in $\Measures(C)$ of Lemma~\ref{sfdc1}.
Let us prove now that
\begin{equation}\label{secondsup}
\sup_{r\in {\cal R}}\int_0^h\int_C r_t\,\de \Sigma_t\,\de t=\int_0^h\sfd_{\rmC^1}(\Sigma_t^1,\Sigma_t^2) \,\de t.
\end{equation}
The inequality $\leq$ is obvious, since $r_t$ is an admissible function in the definition of $\sfd_{\rmC^1}$ for
any $t\in [0,h]$. To prove the converse, we apply a measurable selection argument: 
since $Y$ is separable, it is easily seen that
$$
B\coloneqq
\left\{v\in\rmC_b(Y)\colon v\in\rmC^1_b(Y),\,\, \mathrm{Lip}(v)\leq 1\right\}
$$
is a Borel and separable subset of $\rmC_b(Y)$,
and that $t\mapsto\sfd_{\rmC^1}(\Sigma_t,\Sigma_t')$ is a Borel function. Then, 
for $\delta>0$ fixed we consider the set
$$
\Gamma\coloneqq \left\{ (t,v)\in [0,h]\times\rmC_b(Y)\colon
v\in B,\,\,\sfd_{\rmC^1}(\Sigma_t^1,\Sigma^2_t)<\delta+\int_Y v\,\de\Sigma_t
\right\}
$$
which is measurable, thanks to the above-mentioned properties, with respect to the product of the Borel
$\sigma$-algebras. Then, a measurable selection theorem \cite[Theorem~6.9.1]{Bogachev} grants the existence of a Borel selection
map $\bar r\colon[0,h]\to B$, satisfying $(t,\bar r_t)\in\Gamma$ for a.e.~$t\in [0,h]$. Since, by construction,
$\bar r\in {\cal R}$, it follows that
$$
\sup_{r\in {\cal R}}\int_0^h\int_C  r_t\,\de\Sigma_t\,\de t\geq \int_0^h\int_C  \bar r_t\,\de\Sigma_t\,\de t
\geq
-\delta \int_0^h|\Sigma_t|(C)\,\de t+\int_0^h\sfd_{\rmC^1}(\Sigma_t^1,\Sigma_t^2) \,\de t.
$$
Since $\delta>0$ is arbitrary, this proves \eqref{secondsup}.

Combining \eqref{firstsup} and \eqref{secondsup} we obtain
\begin{equation}\label{thirdsup}
\int_0^h\sfd_{\rmC^1}(\Sigma_t^1,\Sigma_t^2) \,\de t\le\int_0^h \int_C |\mathrm Dg_t(b_t)| \,\de|\Sigma^2_t|\,\de t.
\end{equation}
It remains to estimate from above the right hand side in \eqref{thirdsup}. 
Recalling that the norm on $Y$ is given by $\|y\|_Y=\|(x,\sigma)\|_Y = |x| + \|\sigma\|_{\rm BL}$ and using \eqref{inisup}, we have
\begin{eqnarray*}
&&\int_0^h \int_C |\mathrm D g_t (b_t)|\, \de |\Sigma^2_t|\de t  
\leq \int_0^h \int_C \|\mathrm D g_t\|_{\mathcal L(Y_0,Y)} \|b_t(y)\|_Y \,\de|\Sigma^2_t|(y)\de t \\
&\leq& \mathrm e^{L_*T} \int_0^h (h-t) \int_C \Big \| \Big(0, \Big( J\ast\Sigma_t(x,\cdot) - \int_U J\ast\Sigma_t(x,w)\,\de\sigma(w)\Big)\sigma\Big) \Big \|_Y\,\de |\Sigma^2_t|(y)\de t\\
&\leq&  \mathrm e^{L_*T} \int_0^h (h-t) \int_C \Big\| \Big( J\ast\Sigma_t(x,\cdot) - \int_U J\ast\Sigma_t(x,w)\,\de \sigma(w)\Big)\sigma\Big\|_{\rm BL} \,\de |\Sigma^2_t|(y)\de t.
\end{eqnarray*}
Moreover, \eqref{eq:3} gives
$$
\bigg\| \Big( J\ast\Sigma_t(x,\cdot) - \int_U J\ast\Sigma_t(x,w) \,\de \sigma(w)\Big)\sigma\bigg\|_{\rm BL} \leq 2 
\int_U | J\ast\Sigma_t(x,u) |  \,\de\sigma(u),
$$
and
$$
| J\ast\Sigma_t(x,u) | \leq L_J \sfd_{\rmC^1}(\Sigma_t^1,\Sigma_t^2),
$$
so that, with $S=\sup_t|\Sigma_t^2|(C)$, we get
\begin{equation}\label{fourthsup}
\int_0^h \int_C |\mathrm D g_t (b_t)| \,\de|\Sigma^2_t|\,\de t\leq 2S\mathrm e^{L_*T}L_J h\int_0^h
\sfd_{\rmC^1}(\Sigma_t^1,\Sigma_t^2)\,\de t.
\end{equation}
Combining \eqref{thirdsup} and \eqref{fourthsup} we obtain
$$
\int_0^h \sfd_{\rmC^1}(\Sigma_t^1,\Sigma^2_t)\,\de t   
\leq  2S \mathrm e^{L_*T} L_J h \int_0^h \sfd_{\rmC^1}(\Sigma_t^1,\Sigma^2_t)\,\de t. 
$$
For $h>0$ small enough such that $2S \mathrm e^{L_*T} L_J h<1$, one has
$$
\sfd_{\rmC^1}(\Sigma_t^1,\Sigma^2_t) =0, \qquad \text{for all $0\leq t \leq h$,}
$$
so that the the curves $\Sigma^1$ and $\Sigma^2$ coincide in $[0,h]$ and the proof is achieved
by repeating this argument finitely many times.
\end{proof}

\subsection{Uniqueness by superposition}

In this section we prove uniqueness of Eulerian solutions, as defined in Definition~\ref{def:main}, under the same
assumptions of Theorem~\ref{thm:main}, dealing with Lagrangian solutions. 
  In particular, we require the sole Lipschitz continuity
  of $e$ and $J$, not requiring the $\rmC^1$ smoothness
  of $J(\cdot,u,x',u')$ and 
$e(\cdot,u)$ of Theorem~\ref{thm:uni}. 
The proof covers the more general
setting of interaction systems in Banach spaces of Section~\ref{subsecint}, see Theorem~\ref{prob:mainBanach} for the
existence and stability of Lagrangian solutions. 

Our main tool in the proof is the
so-called superposition principle: it allows to lift solutions to the continuity equation to
probability measures on paths, thus recovering an (extended) Lagrangian representation; 
the principle, which remarkably works under no regularity assumption on the vector field,
 has by now many versions, see for instance \cite[Theorem~8.2.1]{AGS2008} in Euclidean
spaces and \cite{Smirnov} in the context of the theory of currents.
Here we consider the case when the state space is a separable Banach space.

\begin{theorem}
Let $(Y,\|\cdot\|_Y)$ be a separable Banach space, let $b\colon (0,T)\times Y\to Y$ be a Borel vector field and 
let $\mu_t\in\Probabilities{Y}$, $t\in [0,T]$, be a continuous curve with
\begin{equation}\label{eq:prix}
\int_0^T \int_Y {\crd  \|b_t\|_Y}\,\de\mu_t\de t<+\infty.
\end{equation}
If
$$\frac{\de}{\de t}\mu_t+\mathrm{div}({\crd  b_t}\mu_t)=0$$
in duality with cylindrical functions $\phi\in\rmC^1_b(Y)$, precisely of the form (here $\langle\cdot,\cdot\rangle$ denotes the
duality map between $Y$ and $Y'$)
$$
\varphi(\langle y,z_1'\rangle,\langle y,z_2'\rangle,\ldots,\langle y,z_N'\rangle)
$$
with $\varphi\in\rmC^1_b(\R^N)$ and $z_1',\ldots,z_N'\in Y'$, then there exists $\eeta\in\Probabilities{\rmC ([0,T];Y)}$
concentrated on absolutely continuous solutions to the ODE $\dot{y}=b_t(y)$
and with $\mathrm{ev}_t(\eeta)=\mu_t$ for all $t\in [0,T]$.
\end{theorem}
\begin{proof}
We start from the version of the superposition principle in the space $\R^\infty$ of all 
sequences $(x_i)$, $i\geq 1$, proved in \cite[Theorem 7.1]{AT2014} by finite-dimensional approximation.
We can endow $\R^\infty$ with the distance $\sfd_\infty(x,y)\coloneqq 2^{-n}\min\{1,|x_n-y_n|\}$, which makes it a complete and separable metric space. 
If $\cc\colon(0,T)\times\R^\infty\to\R^\infty$ is a vector field, with components $c^i$ measurable with the respect to the product of the Borel
$\sigma$-algebras in the domain, and if
$\nu_t\in\Probabilities{\R^\infty}$ satisfy $\int_0^T\int_{\R^\infty} |c^i_t|\,\de\nu_t \,\de t<+\infty$ 
for any $i$ and solve the continuity equation
$$
\frac\de {\de t}\nu_t+\mathrm{div}(\cc_t\nu_t)=0
$$
in duality with $\rmC^1_b$ cylindrical function $\phi$ (i.e., dependent on finitely many coordinates $x_i$), there exists
$\ssigma\in {\Probabilities{\mathrm{C}([0,T];\R^\infty)}}$ such that:
\begin{itemize}
\item[(1)] $\ssigma$ is concentrated on continuous curves $\gamma\colon[0,T]\to\R^\infty$, with absolutely continuous 
components $\gamma^i$ solving the infinite system of ODE $\dot{\gamma}^i=c^i_t(\gamma)$, $i\geq 1$;
\item[(2)] $\mathrm{ev}_t(\ssigma)=\nu_t$ for all $t\in [0,T]$. 
\end{itemize}
Given this basic result, if we strengthen the integrability assumption on $\cc$, by requiring
\begin{equation}\label{eq:prix2}
\int_0^T\int_{\R^\infty} \Psi(\sup_i|c_t^i|)\,\de\nu_t \,\de t<+\infty
\end{equation}
with $\Psi$ a non-decreasing function such that $\Psi(z)/z\to +\infty$ as $z\to\infty$, 
then it is immediately seen that $\ssigma$ is concentrated on a Borel
set $\Theta$ (i.e.~$\ssigma(\mathrm{C}([0,T];\R^\infty)\setminus\Theta)=0$) \nc
made of 
curves $\gamma$ absolutely continuous with respect to the
norm $\|\cdot\|_\infty$, more precisely one has
$$
\sup_i|\gamma^i_s-\gamma^i_t|\leq \int_s^t
\sup_i|c^i_\tau|\circ\gamma_\tau\,\de\tau,\qquad \text{for all $0\leq s\leq t\leq T$ and all $\gamma\in\Theta$,}
$$
$$
\int_0^T\Psi(\sup_i|c^i_\tau|\circ\gamma_\tau)\,\de\tau<+\infty,\qquad \text{for all $\gamma\in\Theta$.}
$$

Now we turn to the case of a separable Banach space $Y$. Thanks to Dunford-Pettis theorem,
applied to the space-time measure $\mu_t \de t$, we can find a non-decreasing function $\Psi\colon [0,+\infty)\to [0,+\infty)$
with $\Psi(z)/z\to +\infty$ as $z\to +\infty$ such that \eqref{eq:prix} improves to
\begin{equation}\label{eq:prix11}
\int_0^T \int_Y \Psi({\crd  \|b_t \|_Y})\,\de\mu_t\de t<+\infty.
\end{equation}

 Let $(z_i')\subset Y'$ as in the proof of {Lemma~\ref{sfdc1}}
and
let us consider the map $E\colon Y\to\ell_\infty\subset\R^\infty$ defined by
$$E(y)\coloneqq (\langle y,z_1'\rangle,\langle y,z_2'\rangle,\ldots).$$
It is immediately seen {from \eqref{goodyi}}
 that the mapping $E$ is an isometry of $Y$ into $\ell_\infty$, hence, $E(Y)$ is closed and separable in $\ell_\infty$;
in addition, \cite[Theorem~6.8.6]{Bogachev} grants that $E$ maps Borel sets of $Y$ into Borel sets of $\R^\infty$, in particular
$E(Y)$ is a Borel set of $\R^\infty$. As a consequence, 
$E^{-1}$ extended to $0$ out of $E(Y)$ is $\mu$-measurable for any $\mu\in\Probabilities{\R^\infty}$.

Let us consider the measures $\nu_t=\PushForward{E}{\mu_t}\in\Probabilities{\R^\infty}$; 
if we define
$$
c^i_t\coloneqq \langle {\crd  b_t},z_i'\rangle\circ E^{-1}
\qquad\text{on $E(Y)$, for $i\geq 1$}
$$
(and equal to $0$ on $\R^\infty\setminus E(Y)$, this extension is irrelevant since the measures $\nu_t$
are concentrated on $E(Y)$) we obtain by construction that the continuity equation holds, with the stronger
integrability condition \eqref{eq:prix2} coming from
  \eqref{eq:prix11},
since $\sup_i|c^i_t|\le \|b_t\|_Y$ (recall that $\Psi$ is non-decreasing).

Then, from the superposition theorem in $\R^\infty$ we obtain a
probability measure\break
 $\ssigma \in
 \Probabilities{\mathrm{C}([0,T]; \R^\infty)}$ 
concentrated on solutions of the ODE $\dot{\gamma}^i_t=c^i_t(\gamma)$ and absolutely
continuous with respect to the norm $\|\cdot\|_\infty$. 
Moreover, since $\nu_t$ are concentrated on $E(Y)$ 
we obtain that, for any $t\in [0,T]$, $\gamma_t\in E(Y)$ for
$\ssigma$-a.e. $\gamma$. In particular, 
restricting $t$ to a countable set, for $\ssigma$-a.e. $\gamma$, one has
\begin{equation}\label{eq:prix1}
\text{ $\gamma_t\in E(Y)$ for any $t\in \Q\cap [0,T]$.}
\end{equation}
{If we fix $\gamma\in\Theta$ (so that $\gamma$ is absolutely continuous with respect to
the $\|\cdot\|_\infty$ norm) and \eqref{eq:prix1} holds, since our choice of the $z_i'$ guarantees that the composition with $E^{-1}$ 
is an isometry between the norm $\|\cdot\|_\infty$ and $\|\cdot\|_Y$, from $\gamma_t\in E(Y)$ for any $t\in\Q\cap [0,T]$
we deduce $\gamma([0,T])\subset E(Y)$: indeed, if $t_n\in \Q\cap [0,T]$ and $t_n\to t$, $\gamma_{t_n}=E(y_n)$,
then the absolute continuity of $\gamma$ yields that $(y_n)$ is a Cauchy sequence in $Y$, hence $y_n\to y$ for some $y\in Y$ and
then $\gamma_t=E(y)\in E(Y)$.}
Therefore, we proved that for all $\gamma\in\Theta$ the property \eqref{eq:prix1} improves to
$\gamma([0,T])\subset E(Y)$, and the same argument shows that
the transformed curve $E^{-1}\gamma$ in $Y$ is absolutely continuous.

The ODE $\dot{\gamma}^i_t=c^i_t(\gamma)$ becomes, for the transformed curve $y_t=E^{-1}\gamma_t$,
$$
\frac{\de}{\de t}\langle y,z_i'\rangle=\langle b_t\circ y,z_i'\rangle, \qquad\text{for all $i\geq 1$,}
$$
that, since the $z_i'$ separate points in $Y'$, allows to conclude that $\dot{y}=b_t\circ y$. 

Let us now consider the map 
$\tilde E\colon \rmC([0,T];Y)\to \rmC([0,T];\R^\infty)$, $\tilde E(\gamma)\coloneqq E\circ\gamma$,
which naturally extends $E$ to the corresponding complete and
separable spaces of curves. Since $\tilde E$ is continuous and
injective, it maps Borel sets into Borel sets.
It follows that the inverse map
$E_*$, $\gamma\mapsto y=\tilde E^{-1}\gamma$, arbitrarily defined 
to a constant on the
Borel and $\ssigma$-negligible set 
$$\rmC([0,T];\R^\infty)\setminus \tilde E\big(\rmC([0,T];Y)\big)
\subset \rmC([0,T];\R^\infty)\setminus \Theta,$$ 
is Borel.

To conclude, having set $\eeta=(E_*)_\# \ssigma$, from $\mathrm{ev}_t(\ssigma)=\nu_t$
one obtains $\mathrm{ev}_t(\eeta)=\mu_t$ for all $t\in [0,T]$. 
The measure $\eeta$ then satisfies all stated properties.\qedhere
\end{proof}

\begin{theorem}\label{thm:uni2}
Let $\bar\Sigma\in\Probabilitiesone{C}$, $f\colon C\times C\to Y$ be $L$-Lipschitz, and let $b_{\Sigma}$ be defined as
in \eqref{eq:recall}. Then there is a unique $\Sigma\in\rmC([0,T];\Probabilitiesone{C})$ with $\Sigma_0=\bar\Sigma$, satisfying
$$
\int_C \phi(t,y)\,\de\Sigma_t(y)- \int_C \phi(0,y)\,\de\Sigma_0(y)= \int_0^t \int_C \Big(\partial_s \phi(s,y)+
\mathrm D\phi(s,y)(b_{\Sigma_s}(y))\Big)\,\de\Sigma_s(y) \de s
$$
for all $\phi\in\rmC^1_b([0,T]\times Y)$.
\end{theorem}
\begin{proof}
Let us consider Eulerian solutions $\Sigma^1$ and $\Sigma^2$ starting from the same initial datum $\bar\Sigma$ and let us
denote $b^1(t,y)$ e $b^2(t,y)$ the respective velocity fields and $\bY^1(t,y)$, $\bY^2(t,y)$ be the respective flow maps. By applying the
superposition theorem to $b^i$ (extended with the 0 value to $(0,T)\times (Y\setminus C)$), 
$\Sigma^i$ we obtain that $\Sigma^i_t=\mathrm{ev}_t(\eeta^i)$ for suitable
$\eeta^i\in\Probabilities{\rmC([0,T];Y)}$ concentrated on absolutely continuous in $[0,T]$ solutions
to the Cauchy problem
$$
\dot{y}=\bb_{\Sigma^i}(t,y)\quad \text{in $(0,T)$.}
$$
On the other hand, since from Theorem~\ref{thm:Brezis} we know that the solution to the Cauchy problem 
is unique, the conditional probabilities $\eeta^i_y$ of $\eeta^i$ given the initial condition $y(0)=y$ have to be Dirac
masses, precisely $\eeta^i_y=\delta_{\bY^i(\cdot,y)}$.
It follows that $\Sigma^i_t=\PushForward{\bY^i(t,\cdot)}{\bar\Sigma}$ for all $t\in [0,T]$. An application of the stability estimate \eqref{eq:stability} in Theorem \ref{prob:mainBanach} {\crd  yields uniqueness}. \qedhere
\end{proof}

\appendix
\section{Calculus in Banach spaces} \label{app:calc}
We adapt  some basic calculus notions in Banach spaces to our framework, where the
domain $C$ of the functions we wish to differentiate is a convex subset of a normed space $E$ but need not be open; in this
case we denote by $E_C$ the vector space $\R(C-C)$. For $c\in C$, we shall instead denote by $E_c$ the convex cone of
directions
$$
E_c\coloneqq\R_+ (C-c)\subset E_C.
$$

\subsection{Differentiation}\label{sec:diff}
We first introduce a notion of multivalued Fr\'echet
differential, adapted to functions defined on convex sets.

\begin{definition}[Multivalued \Frechet differential]\label{def:FD}
Let $E, \, F$ be normed vector spaces, $C\subset E$ convex, and let $f\colon C \to F$ be a map.  
We say that $f$ is \Frechet differentiable at $c\in C$ if there exists 
$L \in \mathcal L(E_C,F)$ such that
\begin{equation}\label{eq:frechet}
\lim_{C\ni c'\to c} \frac{\|f(c')- f(c) - L(c'-c)\|_F}{\|c'-c\|_E}=0.
\end{equation}
\end{definition}
This notion is too strong for some applications, since the natural domain of $L$ should be only the closure of the
cone $E_{c}$, on which $L$ is uniquely determined by
\eqref{eq:frechet}. 
We denote the F-differential of $f$ in $c \in C$
$$
\mathrm Df(c) = \{ L \in \mathcal L(E_C,F): L \mbox{ fulfills } \eqref{eq:frechet} \},
$$
and, if $E_c$ is not dense in $E_C$, the map $\mathrm D$ is multivalued (as in the case of
the subdifferential in convex analysis). 
By density, each $L \in \mathrm D f(c)$ uniquely extends to an operator in
$\mathcal L(\overline{E}_C,F)$. Hence, the F-differential $\mathrm Df(c)$ is 
a closed convex subset of $\mathcal L(\overline{E}_C,F)$.
For any $e \in E_C$ we denote 
$$
\mathrm Df(c)e=\mathrm Df(c)(e) = \{ L(e)=Le: L \in \mathrm Df(c)\}.
$$
If $e \in \overline{E}_c$ then $\mathrm Df(c)e$ is a singleton and in this case, with a slight abuse of notation, we may use $\mathrm Df(c)e$ instead
of $Le$ for any $L \in \mathrm Df(c)$.

A straightforward consequence of \eqref{eq:frechet} is the chain rule for curves: if $t\mapsto c(t) \in \rmC^1([0,T];C) $ is a differentiable map, then $c'(t)\in\overline{E}_{c(t)}\subset\overline{E}_C$ and
\begin{equation}\label{eq:chain1}
\frac{\de}{\de s}f(c(s))\biggr\vert_{s=t}=\mathrm Df(c(t))(c'(t)),
\end{equation}
whenever $f$ is \Frechet differentiable at $c(t)$ (we remark the abuse of notation mentioned above). 

For  \Frechet differentiability
we can also adopt the handy notation 
$$f(c')- f(c) - \mathrm{D} f(c)(c'-c) = o(\|c'-c\|_E).$$

If we consider the particular case when $E=F$ and $f\colon C\to C$, obviously any $L\in\mathrm D f(c)$ is a linear operator
from $E_C$ to $\overline{E}_C$.
 With these notions, the proof of the following chain rule is standard.

\begin{theorem}[Chain Rule]\label{chainrule}
Suppose $f\colon C  \to C$ and $g\colon C  \to \R$ are \Frechet differentiable, respectively at $c\in C$ and at $d=f(c)\in C$. 
Then $g \circ f\colon C\to\R$ is also \Frechet differentiable at $c$ and 
$$
\mathrm D (g \circ f)(c) \subset \mathrm D g (d)\circ \mathrm D f(c),
$$
where $\mathrm D g (d)\circ \mathrm D f(c) = \{ M \circ L: M\in \mathrm D g (d), L \in \mathrm D f(c) \}$.
\end{theorem} 
Notice that if $e\in E_c$ then $\mathrm{D}f(c)(e)$ is unique (i.e., it does not depend on the choice of $L$) and belongs to $\overline{E}_d$. Therefore, also the choice of the element $M\in\mathrm{D}g(d)$ is irrelevant.
\begin{proof} It suffices to write
$$
g\circ f(c')-g\circ f(c)= g(d+\mathrm D f(c)(c'-c)+R(c'))-g(d)
$$
with $\|R(c')\|_E=o(\|c'-c\|_E)$ as $c'\to c$,  
and the expansion $g(d')-g(d)=\mathrm Dg(d)(d'-d)+o(\|d'-d\|_E)$ with 
$d'=d+\mathrm D f(c)(c'-c)+R(c')$.
\end{proof}

We say that $f$ is of class $\mathrm{C}^1$ and we write $f \in\rmC^1(C;F)$ if $f$ is \Frechet differentiable at each $c\in C$, 
and there exists a selection $L(c) \in \mathrm D f(c)$ for all $c\in C$, such that
\begin{equation}\label{contsel}
\text{$c\mapsto  L(c)$ is continuous from $C$ to $\mathcal L(E_C,F)$}
\end{equation}
with $\mathcal L(E_C,F)$ endowed with the distance induced by
the operator norm.
To conform with the classical continuous Fr\'echet differentiability, below we may write with a slight abuse of notation that 
\begin{equation}\label{contsel2}
\text{$c\mapsto \mathrm D f(c)$ is continuous from $C$ to $\mathcal L(E_C,F)$}
\end{equation}
to actually mean $f \in\rmC^1(C;F)$ as in \eqref{contsel}.

In the context of Theorem~\ref{chainrule}, by choosing a continuous selection for $\mathrm{D}f(c)$ and for $\mathrm{D}g(d)$, we have a continuous selection for the F-differential of the composition at $c$, $\mathrm{D}(g\circ f)(c)$, thus granting the $\rmC^1$ regularity of $g\circ f$.

\begin{definition}[\Gateaux differentiation]
Let $E, \, F$ be normed vector spaces, $C\subset E$ convex, and let $f\colon C \to F$ be a map.  
We say that $f$ is \Gateaux differentiable at $c\in C$ if the directional right derivatives
$$
\de f(c,e)\coloneqq \lim_{h \to 0^+} \frac{ f(c+ h e)-f(c)}{h}
\quad\text{
exist in $F$ for all $e\in E_{c}$.}
$$
\end{definition}

Of course, \Frechet differentiability at $c$ implies \Gateaux differentiability at $c$, with $\de f(c,e)=\mathrm D f(c)(e)$
for all $e\in E_c$.
In connection with the differentiability properties of flow map, it is useful to establish the converse implication.

\begin{lemma}[Criterion for $\rmC^1$ regularity]\label{final_lemma}
Let $f:C\to F$ be a continuous map and assume the existence of a continuous operator
$$
C\ni c\mapsto L(c)\in\mathcal L(E_C,F)
$$
such that $\de f(c,e)=L(c)e$ for all $c\in C$ and all $e\in E_{c}$. Then $L(c) \in \mathrm Df(c)$ is an admissible choice in \eqref{eq:frechet} for all $c\in C$, so that
$f\in\rmC^1(C;F)$.
\end{lemma}
\begin{proof} Set $c_t=c+t(c'-c)$ for $t\in [0,1]$ and set $e=c'-c\in E_c$, and notice that $\pm e\in E_{c_t}$ for all $t\in (0,1)$.
Since $t\mapsto f(c_t)$ is continuous in $[0,1]$ and differentiable in $(0,1)$, we can write
\begin{equation}
f(c')-f(c)=\int_0^1\de f(c_t,e)\,\de t=\int_0^1 L(c_t)e\,\de
t.\label{eq:46}
\end{equation}
Now we use that $L(c_t)e/\|e\|_E\to L(c)e/\|e\|_E$ uniformly in $[0,1]$ as $c'\to c$ to obtain \eqref{eq:frechet}.
\end{proof}

\begin{remark}
  \label{rem:tedious}
  \upshape
  Formula
  \eqref{eq:46} also shows that a function $f\in \rmC^1(C;F)$ with
  a uniformly bounded  \Frechet differential selection $L\colon c\mapsto L(c)$ as in \eqref{contsel}, i.e., $L\in \rmC_b(C;\mathcal L(E_C,F))$, satisfies a uniform bound
  \begin{equation}
    \label{eq:47}
    \|f(c)\|_{F}\le \|f(c_0)\|_F+\|c-c_0\|_E\,\sup_{c\in C}\|L(c)\|_{\mathcal L(E_C,F)},
  \end{equation}
    where $c_0$ is a given point in $C$. 
    In particular $f$ has linear growth
  \begin{equation}
    \label{eq:48}
    \sup_{c\in C}\frac{ \|f(c)\|_{F}}{1+\|c-c_0\|_E}
    \le \|f(c_0)\|+\sup_{c\in C}\|L_c\|_{\mathcal L(E_C,F)}.
  \end{equation}
  The left-hand side of \eqref{eq:48} defines 
  a norm in the space $\rmC_{\rm lg}(C,F)$ of continuous functions
  with linear growth (the definition is in fact
  independent of $c_0$).
  The same argument of the proof of Lemma \ref{final_lemma}
  also shows that the graph of the multivalued operator $\mathrm D$ 
  \begin{equation}
    \label{eq:31}
    X\coloneqq \Big\{(f,L)\in 
    \rmC^1(C,F)\times \rmC_b(C,\mathcal L(E_C,F)):
    L(c)\in \mathrm Df(c)\text{ for every $c\in C$}\Big\},
  \end{equation}
  between $\rmC^1(C;F)$ and $\rmC_b(C;\mathcal L(E_C,F))$ is 
  closed, and thus a Banach space,
  with respect to the graph norm
  \begin{equation}
    \label{eq:45}
    \|(f,L)\|_X\coloneqq \sup_{c\in C}\frac{ \|f(c)\|_{F}}{1+\|c-c_0\|_E}+
    \sup_{c\in C}\| L(c) \|_{\mathcal L(E_C,F)}
  \end{equation}
  which in turn is equivalent to $\|f(c_0)\|_F+\sup_{c\in C}\| L(c) \|_{\mathcal L(E_C,F)}$.
  
The closedness of the graph $X$ is equivalent to saying that if $c\mapsto L_n(c)\in \mathrm D f_n(c)$ is a sequence of maps in $\rmC_b(C; \mathcal L(E_C,F))$ uniformly converging to $L \in \rmC_b(C;  \mathcal L(E_C,F))$ and $\lim_{n\to\infty}f_n(c_0)=f(c_0)$ in $F$, then  $f_n$ is also converging uniformly on bounded sets to a function $f \in  \rmC^1(C;F)$  and $L(c)\in \mathrm Df(c)$ for all $c \in C$.
\end{remark}

\subsection{Bochner integration}\label{BI}
Let $(A,\mathcal A, \mu)$ be a $\sigma$-finite measure space and let $E$ be a Banach space. 
A $\mu$-simple function $f\colon A \to E$ is representable as
$$ f = \sum_{n=1}^N \chi_{A_n} e_n, $$
where $e_n \in E$ and $A_n \in \mathcal A$ with $\mu(A_n) < +\infty$.

\begin{definition}[Bochner integral]\label{def:Bochner}
A function $f\colon A \to E$ is $\mu$-Bochner integrable if there exist simple functions $f_n\colon A \to E$ such that 
\begin{itemize}
\item[(i)] $\lim_n f_n = f$ $\mu$-a.e. (strong $\mu$-measurability);
\item[(ii)] $\lim_n  \int_A \| f_n - f\|_E \,\de \mu =0$. 
\end{itemize}
\end{definition}

If $f$ is $\mu$-Bochner integrable then 
$$
\int_A f \,\de \mu  = \lim_n   \int_A f_n \,\de \mu  \in E
$$
exists, is independent of the sequence $(f_n)$, it is called the Bochner integral of $f$, and satisfies
\begin{equation}\label{eq:intstimanorme}
\bigg\|\int_A f \,\de \mu\bigg\|_E \leq \int_A \|f\|_E \,\de \mu.
\end{equation}

We shall use the fact that, in the case when $(A,\sfd_A)$ is a separable metric space and $\mathcal A$
is the Borel $\sigma$-algebra, any continuous function $f\colon A\to E$ is strongly $\mu$-measurable. This is a consequence
of Pettis measurability theorem (see for instance \cite[Chapter 5]{Yosida80}), since $x\mapsto \langle e',f(x)\rangle$ is continuous, hence
$\mu$-measurable, for any $e'\in E'$ (the so-called weak measurability property), and the separability of the range of $f$. 

A simple criterion for Bochner integrability is the following.

\begin{proposition}[Bochner integrability criterion]
A strongly $\mu$-measurable function (as in (i) of Definition~\ref{def:Bochner}) $f\colon A \to E$ is $\mu$-Bochner integrable if and only if 
$$
\int_A \|f\|_E \, \de \mu <+ \infty.
$$
\end{proposition}

Bochner integral commutes in many ways with linear operators, let us illustrate these properties for duality operators, linear
operators and differentiation operators.

(1) For any $x' \in E'$ one has 
$$
\biggl\langle \int_A f \,\de \mu, x'\bigg\rangle_{E \times E'} =  \int_A \langle f , x' \rangle_{E \times E'}\, \de \mu
$$
and an analogous formula with the pre-dual holds if $E$ is a dual Banach space. 

(2) If $f\colon A \to E$ is $\mu$-Bochner integrable and $T \in \mathcal L(E,F)$, then $Tf\colon A \to F$ is also  $\mu$-Bochner integrable and
$$
 T \int_A f \,\de \mu = \int_A Tf \,\de \mu. 
$$

(3) Let $E,\,F$ be normed spaces.  If $(f,L)\colon A \to X$, where $X$ in the Banach space defined in \eqref{eq:31}, where $L(x,\cdot)$ is a continuous selection in $\mathrm Df(x,\cdot)$, is $\mu$-Bochner integrable, then
\begin{equation}\label{exdiffint}
\mathrm D \int_A  f (x,\cdot)\,\de \mu(x) \ni
\int_A L(x,\cdot) \,\de \mu(x)\in \rmC_b(C;\mathcal L(E_C,F)).
\end{equation}
This useful formula explains the correct exchange of \Frechet differentiation and Bochner integration. In particular
it provides, in conjunction with item (1) above, 
the following simple extension of the fundamental theorem of calculus,
for which in fact the simpler Riemann integral would be
sufficient. 
\begin{theorem}\label{fundcalc}
If $f\colon [a,b] \to F$ is continuous, differentiable in $(a,b)$, and $f'$ extends continuously to $[a,b]$, then
$$
f(b)-f(a) = \int_a^b f'(\xi) \,\de\xi.
$$
\end{theorem}

\section{Well-posedness of ODEs in Banach spaces and linearization} \label{sec:regularity}

First we recall Brezis' theorem \cite[Sect.~I.3, Thm.~1.4, Cor.~1.1]{Brezis} on the well-posedness of ODE's in Banach spaces. 
\begin{theorem}
 \label{thm:Brezis}
  Let $(E,\|\cdot\|_E)$ be a Banach space, $C$ a closed convex subset
  of $E$ and let $A(t,\cdot)\colon C\to E$, $t\in [0,T]$, be a family of operators satisfying the
  following properties: 
  \begin{enumerate}
  \item there exists a constant $L\geq 0$ such that
    \begin{equation}
      \label{eq:15}
      \|A(t,c_1)-A(t,c_2)\|_E\le L\|c_1-c_2\|_E\qquad
      \text{for every $c_1,\,c_2\in C$ and $t\in [0,T]$;}
    \end{equation}
  \item for every $c\in C$ the map $t\mapsto A(t,c)$ is continuous in $[0,T]$;
  \item for every $R>0$ there exists $\theta>0$ such that
    \begin{equation}
      \label{eq:16}
      c\in C,\ \|c\|_E\le R\quad
      \Rightarrow\quad
      c+\theta A(t,c)\in C.
    \end{equation}
  \end{enumerate}
  Then for every $\bar c\in C$ there exists a unique curve
  $c\colon[0,T]\to C$ of class $\rmC^1$ satisfying $c_t\in C$ for all $t\in [0,T]$ and
\begin{equation}\label{eq:17}
\frac{\de}{\de t}c_t=A(t,c_t)\quad\text{in }[0,T],\qquad  c_0=\bar c.
\end{equation}
Moreover, if $c^1,\,c^2$ are the solutions starting from the initial data $\bar c^1,\,\bar c^2\in C$ respectively, we have
\begin{equation}\label{eq:18}
\|c^1_t-c^2_t\|_E\le \mathrm e^{Lt}\|\bar c^1-\bar c^2\|_E,\qquad\text{for every $t\in [0,T]$.}
\end{equation}
\end{theorem}

Assume now that $A(t,\cdot)\colon C\to E$ is a family of operators as required by the assumptions of Theorem~\ref{thm:Brezis}, so that the 
flow map $\YY\colon [0,T]\times C\to C$ given by
\begin{equation}\label{ode1}
\frac{\de}{\de t}\YY(t,c)=A(t,\YY(t,c)), \qquad \YY(0,c) =  c,
\end{equation}
is well defined. For simplicity we consider only the flow map starting from $t=0$, but the same results hold for the
full family of transition maps $\YY(t,s,c)$.

Thanks to \eqref{eq:16}, the operators $A$ take their values in $E_C$.
In this section we shall highlight additional conditions on $A$ for the flow map $\YY(t,\cdot)$ 
to be continuously \Frechet differentiable, uniformly in $t\in [0,T]$.  More precisely, we assume that the operators $A(t,\cdot)$ satisfy
\begin{equation}\label{extraA}
\text{$(t,c)\mapsto\mathrm D_{E_C} A(t,\cdot)(c)$ is continuous from $[0,T]\times C$ to $\mathcal L(E_C,\overline{E}_C)$.}
\end{equation}
According to \eqref{contsel2}, by \eqref{extraA} we mean that there exists a selection $L_{A}(t,c) \in \mathrm D_{E_C} A(t,\cdot)(c)$ such that
$(t,c)\mapsto L_{A}(t,c)$ is continuous from $ [0,T] \times C$ to $\mathcal L(E_C,\overline{E}_C)$.

Under this assumption, we prove by classical linearization and stability arguments
the existence of the \Gateaux derivatives, and then their continuity as functions of the point of differentiation,
using eventually Lemma~\ref{final_lemma} to obtain the \Frechet differentiability.

\begin{theorem}\label{thm:extraA}
Under the assumptions of Theorem~\ref{thm:Brezis} and \eqref{extraA}, the map $\YY(t,\cdot)$ is continuously \Frechet differentiable for all $t\in [0,T]$, with a \Frechet differential $\mathrm D\YY$ satisfying:
\begin{itemize}
\item[(i)] $(t,c)\mapsto \mathrm D\YY(t,\cdot)(c)$ continuous from $ [0,T] \times C$ to $\mathcal{L}(E_C,\overline{E}_C)$;
\item[(ii)] $\|\mathrm D\YY(t,\cdot)(c)\|_{{\mathcal L}(E_C,\overline{E}_C)}\leq\mathrm e^{LT}$ for all $(t,c)\in [0,T]\times C$.
\end{itemize}
\end{theorem}
\begin{proof}
For $c\in C$, $e \in E_c$ and $h>0$ sufficiently small, let us consider the finite difference
\begin{equation}\label{defzn}
z_h(t,(c,e))= \frac{\YY(t,c + h e)- \YY(t, c)}{h}.
\end{equation}
From \eqref{ode1} we obtain the equation
\begin{equation}\label{ode2}
\frac{\de}{\de t} z_h(t,(c,e)) =\frac{A(t,\YY(t, c + h e))- A(t,\YY(t,c))}{h}, \qquad z_h(0,(c,e)) = e.
\end{equation}
In view of the well-posedness of \eqref{ode1}, also \eqref{ode2} is well-posed. 
We highlight now some regularity properties of the functions $t \mapsto z_h(t,(c,e))$, which naturally come from 
the assumptions on $A$.

By definition of $z_h(t,(c,e))$ and \eqref{eq:18} we observe that
\begin{equation}\label{eq:bastaquesto}
\sup_{t \in [0,T]} \| z_h(t,(c,e)) \|_E \leq \|e\|_E {\rm e}^{LT},
\end{equation}
where $L$ is the constant of Theorem~\ref{thm:Brezis}(i).
Hence $t \mapsto z_h(t,(c,e))$ are uniformly bounded in $[0,T]$ with respect to $h$.
To compute their limit as $h\downarrow 0$, notice that
an application of Theorem~\ref{thm:Brezis} (or more classical results for linear ODE's)
yields {for all $f\in E_C$} the existence and the uniqueness of a $\rmC^1$ map 
$t \mapsto z(t)\in\overline{E}_C$ 
solving the linear differential equation
\begin{equation}\label{eq:ODEe}
 z'(t) = L_A(t,\YY(t,c)) z(t),\qquad z(0)=f,
\end{equation}
for a continuous selection $L_A(t,\YY(t,c))  \in \mathrm D_{E_C} A(t,\cdot)(\YY(t,c))$, whose existence is granted by \eqref{extraA}.
For later use we denote the solution $z=z_f$ to emphasize its initial datum.

From \eqref{defzn} and \eqref{ode2} we have the identity (with the notation
$z^h(t)\coloneqq z_h(t,(c,e))$) 
\begin{equation}
\begin{split}
\!\!\!\frac{\de}{\de t} (z^h(t)-z_e(t))= h^{-1}[ & A(t,\YY(t, c)+h z^h(t))-A(t,\YY(t,c))-h L_A(t,\YY(t,c))z^h(t)]  \\
&+ L_A(t,\YY(t,c))[z^h(t)-z_e(t)]. \label{luigiexp}
\end{split}
\end{equation}
We recall now that for any $g \in\rmC^1([0,T];E)$ the map $t \mapsto \|g(t)\|_E$ is Lipschitz and satisfies (as a simple consequence of 
the fundamental theorem of calculus and \eqref{eq:intstimanorme})
$$
\bigg|\frac{\de}{\de t} \|g(t)\|_E \bigg|\leq \|g'(t)\|_E, \qquad \text{for almost every $t\in [0,T]$.}
$$
Then,  we obtain the estimates
\begin{equation*}
\begin{split}
\frac{\de}{\de t} \|z^h(t)-  z_e(t)\|_E  \leq  \big \| & h^{-1}[A(t,\YY(t,c)+h z^h(t))-A(t,\YY(t,c))-h L_A(t,\YY(t,c))z^h(t)]  \\
& +L_A(t,\YY(t,c))[z^h(t)-z_e(t)]  \big \|_E\\
\leq  \big \| & h^{-1}[A(t,\YY(t,c)+h z^h(t))-A(t,\YY(t,c))-h L_A(t,\YY(t,c))z^h(t)] \big \|_E \\
&+L\|z^h(t)-z_e(t)  \|_E,
\end{split}
\end{equation*}
since $\|L_A(t,\YY(t,c))\|_{\mathcal L(E_C,\overline{E}_C)}\leq L$, thanks to \eqref{eq:15}.
By Gronwall's inequality and using that $z^h(0)= z_e(0)=e$ we obtain
\begin{eqnarray*}
&&\|z^h(t)-  z_e(t)\|_E \\
&\leq &  \int_0^t  \big \|  h^{-1}[A(s,\YY(s,c)+h z^h(s))-A(t,\YY(s,c))-h L_A(t,\YY(t,c))z^h(s)] 
\big \|_E {\rm e}^{L(t-s)}\, \de s.
\end{eqnarray*}
By the pointwise limit
$$
\big \|  h^{-1}[A(s,\YY(s,c)+h z^h(s))-A(t,\YY(s,c))-h L_A(t,\YY(t,c))z^h(s)] \big \|_E \to 0 \qquad \text{for $h \to 0$}, 
$$
and by the dominated convergence ensured by assumption \eqref{eq:15} and the uniform boundedness in
\eqref{eq:bastaquesto}, we conclude that
$$
\lim_{h\to 0} \sup_{t \in [0,T]} \|z^h(t)-  z_e(t)\|_E  =0.
$$
Moreover, with a similar argument as above we have
$$\frac{\de}{\de t} \|z_e(t)\|_E  \leq  L \|z_e(t)  \|_E,$$ 
and again by Gronwall's inequality and $ z_e(0)=e$
\begin{equation}\label{boundedness}
 \|z_e(t)\|_E  \leq \|e\|_E e^{LT}.
\end{equation}
We shall denote now by $L(t,c)$ the  linear operators induced by \eqref{eq:ODEe} by setting $L(t,c)f =z_f(t)$ for any $f \in E_C$ (not to be confused with  $L_{A}(t,c) \in \mathrm D_{E_C} A(t,\cdot)(c)$). Notice that in view of \eqref{boundedness} the operators  $L(t,c) \in\mathcal L(E_C,\overline{E}_C)$ are  in fact  uniformly bounded.
With all of this we proved that $\YY(t,\cdot)$ is \Gateaux differentiable (uniformly in $t\in [0,T]$) and that
$$
\de\YY(t,c,e)=L(t,c)e,\qquad\text{for every $e\in E_c$, $t\in [0,T]$,}
$$
In order to improve from G-differentiability to F-differentiability, we apply Lemma~\ref{final_lemma} if we were able to show that the operator $(t,c)\mapsto L(t,c)$ is continuous from $ [0,T]\times C$ to $\mathcal L(E_C,\overline{E}_C)$. Since uniform continuity with respect to $t$ is obvious, the continuity with respect to $c$ can be again shown by considering, for $c,\,c'\in C$, the operator 
$$
Z(t,c,c') \coloneqq L(t,c) - L(t,c')
$$
Indeed, for $c,\,c'$ fixed, we have the linear differential equation in the space $\mathcal L(E_C,\overline{E}_C)$  
\begin{eqnarray*}
\frac{\de}{\de t} Z(t,c,c') =  L_A(t,\YY_t(c')) \circ Z(t,c,c') + 
[  L_A(t,\YY_t(c'))-  L_A(t,\YY_t(c))] \circ  L(t,c),
\end{eqnarray*}
with the initial condition $Z(0,c,c')=0$; here, we denoted $\YY_t(c)=\YY(t,c)$. 
Under the assumption \eqref{extraA}
and by a similar argument based on the Gronwall inequality as above, since,
for $c$ fixed, $ L_A(t,\YY_t(c'))\to \mathrm  L_A(t,\YY_t(c))$ as $\|c-c'\|_E \to 0$,
we can conclude that
$Z(t,c,c')\to 0$ as $\|c-c'\|_E\to 0$ uniformly in $[0,T]$, 
and therefore the continuity of $c\mapsto L(t,c)$. It follows that $\YY_t$ is of class $\rmC^1$, uniformly
with respect to $t\in [0,T]$, with a \Frechet differential at $c$ given by $L(t,c)$.
\end{proof}

\noindent
{\bf Acknowledgements.} L.A., M.M., and G.S.~are members of the Gruppo Nazionale per l'Analisi Matematica, la Probabilit\`a e le loro Applicazioni (GNAMPA) of the Istituto Nazionale di Alta Matematica (INdAM).
L.A.~and G.S.~acknowledge the support of the MIUR PRIN 2015 project
\emph{Calculus of Variations}. 
G.S.~acknowledges the support of IMATI-CNR and of the project 
\emph{Variational evolution problems and optimal transport} by
Cariplo foundation and Regione Lombardia.
M.F.~and M.M.~acknowledge the support of the ERC Starting grant \emph{High-Dimensional Sparse Optimal Control} (Grant agreement no.~306274) and of the DFG Project \emph{Identifikation von Energien durch Beobachtung der zeitlichen Entwicklung von Systemen} (FO 767/7).

\end{document}